%% file: Flandoli-Olivera-Simon-2020-with-small-changes.tex
\documentclass[onefignum,onetabnum]{siamart190516}

\input{ex_shared}

\ifpdf
\hypersetup{
  pdftitle={},
  pdfauthor={}
}
\fi

\begin{document}

\maketitle

\begin{abstract}
We consider an interacting particle system  modeled as a system of  $N$ stochastic differential equations driven by Brownian motions. We prove that the (mollified) empirical process converges, uniformly in time and  space variables, to the solution of the two-dimensional Navier-Stokes equation written in vorticity form. 
 The  proofs  follow a semigroup  approach.  
\end{abstract}



\section{Introduction}

The  main goal of this paper is to provide a stochastic particle approximation of the two-dimensional Navier-Stokes equation. Precisely,  we consider the following classical Cauchy problem which describes the evolution of the velocity field
$u:\RR_+ \times \RR^2 \to \RR^2$ of an incompressible fluid with kinematic
viscosity coefficient $\nu >0$: for any $(t,x) \in \RR_+ \times \RR^2 $,
\begin{equation}\label{Navier}
 \left \{
\begin{aligned}
    \partial_t u(t, x) &=\nu  \Delta u(t, x) - \big[u(t, x) \sca \nabla\big] u(t, x)  - \nabla p(t, x)
    \\\text{ div } u(t, x)&=0  \vphantom{ u^{\rm ini}}
    \\    u(0,x)&=  u^{\rm ini}(x),
\end{aligned}
\right .
\end{equation}
 where: $\sca$ denotes the standard Euclidean product in $\RR^2$ ;
the unknown quantities are the velocity $u(t, x) = (u_{1}(t, x), u_{2}(t, x)) \in \RR^2$ of
the fluid element at time $t$ and position $x$ and the pressure $p(t, x)\in\RR$. Such equations are attracting the attention of
a large scientific community, with a large amount of publications in the literature. Since this system is very famous, we do not comment here on its derivation and rather refer to the monographs \cite{Soh} and \cite{Tema}. For recent developments, see  also  \cite{Lema}. 

The associated (scalar) \textit{vorticity field} $ \xi= \partial_{1}u_{2} - \partial_{2}u_{1}:  \RR^2 \to \RR $ 
satisfies a remarkably simple equation of \textit{convection-diffusion propagation}, namely
\begin{equation}\label{Vorty}
\partial_{t} \xi  + u \sca \nabla \xi=\nu \Delta \xi, \qquad x\in \mathbb{R}^{2},  t>0.
\end{equation}
The velocity field $u(t, x)$ can be reconstructed from the vorticity distribution $\xi(t, x)$ by the
convolution with the \textit{Biot-Savart kernel} $K$ as:
\begin{equation}
u(t,x)= \big(K \ast \xi(t,\cdot)\big)(x)= \frac{1}{2\pi} \int_{\mathbb{R}^{2}} \frac{(x-y)^{\bot}}{|x-y|^{2}} \ \xi(t,y) dy 
\end{equation}
where $(x_{1},x_{2})^{\bot}:=(-x_{2}, x_{1})$. It is well-known (see \cite[Lemma 1.1]{Mela2}) that there is a constant $c_K >0$ such that: for any $\xi \in \LL^1(\RR^2)\cap \LL^\infty(\RR^2)$ 
\begin{equation}
\label{eq:ck}
\|K \ast \xi \|_{\LL^\infty} \leq c_K \big( \| \xi \|_{\LL^1} +\|\xi\|_{\LL^\infty}\big),
\end{equation} where $\|\cdot \|_{\LL^p}$ denotes the usual $\LL^p(\RR^2)$ norm. The proof of \eqref{eq:ck} simply follows from expanding the convolution and dividing $\mathbb{R}^2$ into two parts, the first one containing the points $(x,y)$ where $|y-x|\leqslant 1$, the second one being its complement.

There is a huge literature on that model: for instance, the Cauchy problem (\ref{Vorty}) for  an initial data in $\mathbb{L}^{1}(\mathbb{R}^2)$ (also $\mathbb{L}^{1}\cap \mathbb{L}^{p}$) was studied for instance in \cite{Ben}.  The existence of solutions of (\ref{Vorty}) for the case of an initial  finite measure  was  proved in 
\cite{Giga} and \cite{Kato}. Uniqueness in that case is a much more difficult problem:  it is shown in \cite{Giga} that the solution is unique if the atomic part of the initial vorticity is sufficiently small.  This last restriction has been removed recently
in \cite{Gallay}: there, the authors  obtain uniqueness when the initial data belongs to  the space of finite measures.

The question of a particle approximation to the $2d$ Navier-Stokes equation has
been already considered in the literature, as recalled in more detail in Section \ref{sec:related} below. The
aim of this paper is to provide a new rigorous approximation of the vorticity
field $\xi$ by stochastic particle systems, stronger than others:  contrary to the previous works  
where only the empirical measure of the density of particles is shown to converge, here we also prove that  a \textit{mollified empirical measure} converges \textit{uniformly}. More precisely, we  consider the  $N$-particle dynamics described, for each $N\in\mathbb{N}$, by the following system  of coupled stochastic differential equations in $\mathbb{R}^{2}$: for any $i=1,\dots,N$,
\begin{equation}
dX_{t}^{i,N}=  F\Big( \frac{1}{N}\sum_{k=1}^{N} (K\ast V^{N})(X_{t}^{i,N} -X_{t}^{k,N})\Big)\; dt + \sqrt{2 \nu} \; dW_{t}^{i} \label{itoassS0}%
\end{equation}
where: \begin{itemize}
\item for a given $M>0$ chosen ahead (see Theorem \ref{Thm 1} below), the function $F$ is given by
\begin{equation}\label{eq:defF2}
F  : \begin{pmatrix} x_1 \\ x_2 \end{pmatrix}  \mapsto \begin{pmatrix} (x_1\wedge M)\vee(-M)  \\ (x_2 \wedge M)\vee(-M) \end{pmatrix} ;
\end{equation}
\item $\{W_{t}^{i}, \; i\in\mathbb{N}\}$ is a family of independent standard Brownian motions on $\RR^2$ defined on a filtered probability space $\left(\Omega,\mathcal{F},\mathcal{F}_{t},\PP\right) $ ;
\item the \textit{interaction potential} $V^N: \RR^2\to \RR_+$ is continuous and will be specified later on.
 \end{itemize} Finally, $ \ast $ stands for the standard convolution product, and $\wedge$ (resp. $\vee$) is the usual notation for the minimum (resp. maximum) of two real numbers.
 
Now let us define the \textit{empirical process} of this particle system as \[S_{t}^{N}:=\frac{1}{N} \sum_{i=1}^{N}\delta_{X_{t}^{i,N}}, \] which is a (scalar) measure-valued
process associated to the $\RR^2$-valued processes $\{t\mapsto X_{t}^{i,N}\}_{i=1,...,N}$. Above,
$\delta_{a}$ is the delta Dirac measure concentrated at $a \in \RR^2$. For any test function $\phi: \RR^2 \to \RR$, we use the standard notation \[\langle S_t^N,\phi \rangle := \frac{1}{N}\sum_{i=1}^N \phi(X_t^{i,N}). \] Our interest lies in the investigation of the dynamical process
$t\mapsto S_{t}^{N}$ in the large particle limit $N\rightarrow\infty$.

The dynamics of the empirical measure is determined by the It\^{o} formula, which reads as follows: for any test function $\phi:\RR^2 \to \RR$ of class $C^2$, the empirical measure $S_{t}^{N}$ satisfies%
\begin{align}
\left\langle S_{t}^{N},\phi\right\rangle =  \left\langle S_{0}^{N}%
,\phi\right\rangle& +\int_{0}^{t}\left\langle S_{s}^{N}, F\big(K \ast V^{N}\ast
S_{s}^{N}\big) \sca \nabla\phi\right\rangle ds \notag\\ 
& + \nu  \int_{0}^{t}\left\langle S_{s}^{N},\Delta\phi\right\rangle
ds+\frac{\sqrt{2\nu}}{N}\sum_{i=1}^{N}\int_{0}^{t}\nabla\phi(X_{s}^{i,N}) \sca \; dW_{s}^{i}.
\label{ident fro S^N}%
\end{align}
Our main result is the uniform convergence (in the space  and time variables) of the \textit{mollified empirical measure} \[g_t^{N}:= V^{N} \ast S_t^{N}: x \in \RR^2 \mapsto \int_{\RR^2} V^N(x-y) dS_t^N(y)\]  to the solution of the Navier-Stokes equation written in vorticity form, given below in Theorem \ref{Thm 1}. Note that this probability measure is more regular than $S_t^N$, and its nicer properties allow us to obtain better convergence results.
 To prove the latter, we follow the  new approach presented  in \cite{flan}
and then in \cite{flan4,flan3,Simon-O}, based on semigroup theory.  Our source of inspiration has been the works of Oelschl\"ager 
\cite{Oel1} and Jourdain and M\'el\'eard \cite{Mela}, where stochastic approximations of PDE's are investigated. We assume that the  initial vorticity  satisfies $\xi^{\rm ini}\in \LL^{1}(\RR^2)\cap \LL^{\infty}(\RR^2)$, but we believe that our approach can be adapted for more irregular initial data,  for instance when $\xi^{\rm ini}$ belongs to $\LL^{1}(\RR^2)\cap \LL^{p}(\RR^2)$ with $ p \in (2, \infty)$. 

Let us also note that a similar strategy based on a mild formulation for the empirical measure (\textit{not} the mollified one) has recently been worked out in \cite{BC}, where the authors prove a law of large numbers for weakly interacting particles driven by independent Brownian motions, under weak assumptions on the initial condition.

\subsection{Related works}

\label{sec:related}

Rigorous derivations of particle approximations to the $2d$ Navier-Stokes equation have already been investigated in the literature.
Chorin in \cite{Chorin1} (see also \cite{Chorin})  proposed a heuristic probabilistic algorithm to
numerically simulate the solution of the Navier-Stokes equation in two dimensions,
by approximating the (scalar) vorticity function, involving cut-off 
kernels,  by random interacting ``point vortices''. The convergence of Chorin’s vortex method was  mathematically proved in 1982, for instance by Marchioro and Pulvirenti \cite{Marchiolo}, who interpreted the vortex equation in two dimensions with bounded and integrable initial condition as a generalized
McKean-Vlasov equation. Simultaneously, several authors obtained convergence proofs for Chorin's algorithm, see for instance Beale and Majda \cite{BM1,BM2} and
Goodman \cite{Good}. Finally, a rate of convergence result was obtained by Long in \cite{Long}. Later M\'el\'eard \cite{Mela2,Mela3} improved the results  and showed the convergence in the path
space of the empirical measures of the interacting particle system. Fontbona  \cite{Fontona} then
generalized that result in dimension $d=3$.

In addition, following the probabilistic interpretation of \cite{Marchiolo}, a series of papers investigates in detail the  \textit{propagation of chaos} property. In 1987, Osada \cite{Osa} proved such a result for an interacting particle system 
which approximates the solution of the McKean-Vlasov equation, without cut-off, by an analytical method based on 
generators of generalized divergence form, but only for large viscosities and 
bounded density initial data. The convergence of empirical measure and propagation of chaos have  then been
considered under more general assumptions and with innovative techniques of
entropy and Fisher information by Fournier, Hauray and Mischler
\cite{Fournier}. Finally we mention that recently, Jabin and Wang
\cite{Jabin}  showed that  a mean field approximation  converges to the solution
of the Navier-Stokes equation written  in vorticity form, and they are able to obtain quantitative optimal convergence rates for all finite marginal distributions of particles.

Besides, let us note that  Marchioro and Pulvirenti  \cite{Marchiolo} wished to describe
a unified approach for both Navier-Stokes and Euler equations, and for that reason they did not
fully exploit the stochastic nature of the Navier-Stokes equation, which is exactly what we are doing here. In fact,  we strongly exploit the Brownian perturbation of the system, and therefore we  
cannot cover the  results obtained for the Euler equation as  in \cite{Marchiolo}.

\subsection{Notations and results}

Before concluding the introduction, let us state the main results of this work. We first need to introduce some of our notations, which are listed below: 
\begin{itemize}
\item For any measure space  $(S,\Sigma,\mu)$, the standard $\mathbb{L}^p(S)$--spaces of real-valued functions, with $p \in [1,\infty]$, are provided with their usual norm denoted by $\| \cdot \|_{\mathbb L^p(S)}$, or $\|\cdot \|_{\LL^p}$ whenever the space $S$ will be clear to the reader. With a little abuse of notation, and as soon as no confusion regarding the space $S$ arises, we denote by $\langle f,g\rangle$ the inner product on $\mathbb{L}^2(S)$ between two functions $f$ and $g$. In more general cases, if the functions take values in some space $X$, the notation will become $\mathbb{L}^p(S ;  X).$ Finally, the norm $\Vert \cdot\Vert_{\mathbb L^p(S) \to \mathbb L^p(S)}$ is the usual operator norm. 
\item For any  $\varepsilon\in \RR, \mm{p\geqslant 1}$, and $d \in \NN$, we denote by
$\HH_{\mm{p}}^{\varepsilon}(  \mathbb{R}^{d})  $ the \textit{Bessel potential space}
\[\HH_{\mm{p}}^{\varepsilon}(  \mathbb{R}^{d}):= \Big\{ u \in \mm{\mathcal{S}'(\mathbb R^d)} \; ; \;  \mathcal{F}^{-1}\Big( \big(1+|\cdot|^{2}\big)^{\varepsilon/2%
}\; \mathcal{F} u(\cdot) \Big) \in \mathbb L^\mm{p}(\mathbb R^d)\Big\},  \]
where $\mathcal Fu$ denotes the \textit{Fourier transform} of $u$. These spaces are endowed with their norm 
\[
\left\Vert u\right\Vert _{\varepsilon,\mm{p}}%
:=\left\Vert \mathcal{F}^{-1}\Big((1+|\cdot|^{2})^{\varepsilon/2%
}\;\mathcal{F} u(\cdot)\Big)\right\Vert _{\mathbb L^{\mm{p}}\left(  \mathbb{R}^{d}\right)  }%
<\infty.
\]
Note that
 \[\left\Vert u\right\Vert _{0,2}=\left\Vert u\right\Vert_{\mathbb{L}^2(\mathbb{R}^d)} \]
and moreover,  for any $\varepsilon\leqslant 0$, we have (using Plancherel's identity and the fact that $(1+|\cdot|)^{\varepsilon/2} \leqslant 1$)
\[
  \left\Vert u\right\Vert _{\varepsilon,2} 
=\left\Vert    (1+|\cdot|^{2})^{\varepsilon/2}\;\mathcal{F} u(\cdot)
\right\Vert_{\mathbb{L}^2(\mathbb{R}^d)} 
\leq \left\Vert \mathcal{F} u
\right\Vert_{\mathbb{L}^2(\mathbb{R}^d)} = \left\Vert u\right\Vert_{0,2} \; .
 \]
\item  Let us now recall the definition of \textit{Sobolev-\mm{Slobodeckij} spaces}. Let  $U$ be a general, possibly non smooth, open set in $\mathbb{R}^d$.  
Let \mm{$p \geqslant 1$}. For any positive integer  $m$ we define
\[
\mathbb{W}^{m,p}(U) := \Big\{ f\in  \mathbb{L}^p(U) \; ; \;  \| f\|_{m,p}:= \sum_{|\mm{s}|\leq m} \| D^\mm{s} f\|_{\mathbb{L}^p(U)}<\infty
    \Big\}. 
\]
For any $\mm{\varepsilon}>0$ \textit{not} an integer, we define 
\[ 
\mathbb{W}^{\mm{\varepsilon},p}(U):=  \Big\{ f\in  \mathbb{W}^{\left[ \mm{\varepsilon}\right],p}(U) \;  ; \;
\| f\|_{\mm{\varepsilon},p}:= \sum_{|\mm{s}|=\left[ \mm{\varepsilon}\right]} \mathcal{I}_\mm{s}(f)
<\infty \Big\}, 
\] where
\[ \mathcal{I}_\mm{s}(f):= \bigg(   \int_{U} \int_{U} \frac{|D^{\mm{s}}f(x)-D^{\mm{s}}f(y)|^p }{|x-y|^{d+(\mm{\varepsilon-\left[ \varepsilon\right]})p}}   dx  dy      \bigg)^{1/p}.  \]
We observe that:  when $U= \mathbb{R}^d$  and $p=2$, the Sobolev space $\mathbb{W}^{\mm{\varepsilon},2}(\mathbb{R}^d)$ and the Bessel space $\HH_{2}^{\mm{\varepsilon}}(\mathbb{R}^d)$ coincide: $\mathbb{W}^{\mm{\varepsilon},2}(\mathbb{R}^d)=\HH_{2}^{\mm{\varepsilon}}(\mathbb{R}^d)$. Moreover,   note that
 for any open set $U$, $\mathbb{W}^{\mm{\varepsilon},2}(U)$  roughly  corresponds to distributions $f$ on $U$ which are restrictions of some $f\in \mathbb{H}^{\mm{\varepsilon}}_2(\mathbb{R}^{d})$, see \cite{Triebel} for instance. \mm{Also we recall that $\mathbb{H}_p^\varepsilon \subset \mathbb{W}^{\varepsilon,p}$ for any $p>1 $ and $\varepsilon \geqslant 0.$}

\item The space of smooth real-valued functions with compact support in $\mathbb{R}^d$ is denoted by $C_0^\infty(\mathbb{R}^d)$. The space of functions of class $C^k$ with $k \in \mathbb{N}$ is denoted by $C^k(\mathbb{R}^d)$. Finally, the space of bounded functions is denoted by $C_b(\RR^d)$.
\end{itemize}
\bigskip

Now, let us give our main assumptions: first, we need to be more precise about the interaction potential $V^N$ ; second, recall that we are interested in the large $N$ limit of the process $t \mapsto S_t^N$, and we therefore need to specify its initial condition, which is  random, and  supposed to be ``almost chaotic'', see point \textit{4.}~in Assumption \ref{assump} below. The expectation with respect to $\PP$ is denoted by $\EE$. We say that a function $f:\RR^2 \to \RR_+$ is a probability density if $\int_{\RR^2} f(x)dx=1$.

\begin{assumption}\label{assump}

We assume that there exists a probability density $V:\RR^2 \to \RR_+$, and a parameter $\beta \in [0,1]$, such that
\begin{enumerate}
\item for any $x \in \RR^{2}$, $V^{N}(x)=N^{2\beta}V(N^{\beta}x)$ ;

\item $V\in C_{0}^{\infty}(\mathbb{R}^{2})$ ;
\item \mm{there exists $p > 2 $ and  $\frac2p < \alpha < 1$ such that, for any $q>0$,}
\begin{equation}
\sup_{N \in \NN} \EE\left[  \big\|  V^{N} \ast  S_{0}^{N} \big\|_{\mm{\alpha,p}}^{\mm{q}}   \right] \ 
<\infty \; ;  \label{initial cond}%
\end{equation}

\item there exists $\xi^{\rm ini} \in \mathbb{L}^1(\RR^2) \cap \mathbb{L}^\infty(\mathbb{R}^2) $ such that the sequence of measures $\{S_0^N\}_N$ weakly converges  to the initial measure $\xi^{\rm ini}(\cdot)dx$, as $N\to\infty$, in probability. 

  \item finally, the  parameters  \mm{$(\beta,\alpha,p)$} satisfy: \begin{equation} \label{eq:alpha} 0 < \beta <  \mm{\frac{1}{4 +2 \alpha-\frac4p}} < \frac14.\end{equation}

\end{enumerate}

\end{assumption}

\begin{remark} In \cite{flan3} the authors provide sufficient conditions for the validity of \eqref{initial cond}. 
To
understand, very roughly, condition \eqref{initial cond}, think to dimension 1 and $\beta=\frac12$:
if we have $N$ points on the real line, distributed very regularly and we
convolve (\textit{i.e.}~observe) them by a smooth kernel $V^{N}$ such that it averages
$\sqrt{N}$ of them, the result of the convolution is a function which does not oscillate too much; opposite to the case in which the concentration of the
kernel $V^{N}$ is such that it averages only very few points, so that the
convolution is exposed to the granularity of the sample, its minor
irregularities and concentrations. Condition \eqref{initial cond} quantifies this control on oscillations.

Some further intuition comes from kernel smoothing algorithms, those which
replace an histogram by a smooth curve; the histogram is based on a partition
of the real line and simply counts the relative frequency of a sample in the
intervals of the partition, kernel smoothing convolves the sample with a
smooth kernel, \textit{e.g.}~a Gaussian kernel with standard deviation $h$. If,
compared to the cardinality $N$ and distribution of the sample, the partition
is made of too small intervals or $h$ is too small, we see an histogram or a
kernel smoothing function which oscillates very much. This happens in
particular when $h$ is of the order of the distance between nearest neighbor
points in the sample. But if we take $h$ much larger, although very small
compared to the full sample, for instance $h\sim N^{-1/2}$ (if the points are
concentrated in a set of size of order one), the graph of the curve given by
kernel smoothing algorithms is not oscillating anymore. 

Let us emphasize, however, that condition \eqref{initial cond} is a joint condition on the size
of the smoothing kernel compared to the cardinality of the sample (the issue
stressed in the previous sentences), but also on the regularity of the sample.
If it has extreme concentrations around some points, the pictures above
change, oscillations may re-appear.
\end{remark}

In all what follows we fix a time horizon $T \geq 0$.

\begin{theorem}
\label{Thm 1}  We assume Assumption \ref{assump}  and we consider the particle system \eqref{itoassS0} with the parameter $M$ which satisfies \begin{equation}\label{eq:assM} M \geqslant c_K\big(1+ \|\xi^{\rm ini}\|_{\mathbb{L}^\infty}\big),\end{equation} where $c_K$ has been defined in \eqref{eq:ck}.

 Then, for every
$\mm{\eta \in(\frac2 p,\alpha)}$, the sequence of processes $\{t\mapsto g_t^{N}=V^N \ast S_t^N\}_{N\in\NN}$ converges in probability with respect to the 
\smallskip

\begin{itemize}

\item weak topology of $\LL^{2}\left(  [0,T]\; ;\; \HH_{\mm{p}}^{\alpha}(
\mathbb{R}^{2})  \right)  $, 

\item strong topology of $C\left( [ 0,T]\; ;\; \mm{\mathbb{W}_{\mathrm{loc}}^{\eta,p}}(
\mathbb{R}^{2})  \right)  $,  
\end{itemize}
as $N\rightarrow\infty$, to the unique weak solution of the   partial differential equation (PDE)
\begin{equation} \left\{\begin{aligned} 
& \partial_{t} \xi+  \mathrm{div}\big(\xi(K\ast \xi) \big) = \nu \Delta \xi \,, \vphantom{\bigg(} \\
&\xi(0,x)=\xi^{\rm ini}(x), \qquad \qquad \quad x \in \RR^2, t>0. \end{aligned}\right. \label{PDEintr2}%
\end{equation}
 Namely, for any real-valued test function $\phi\in C_{0}^{\infty}(\mathbb{R}^{2})$ and any $t\geqslant 0$,
 it holds
\begin{equation}\label{limiteq}
\left\langle \xi(t,\cdot),\phi\right\rangle =\left\langle \xi^{\rm ini},\phi\right\rangle
+  \int_{0}^{t} \big\langle  \xi(s,\cdot),  (K\ast \xi) (s,\cdot)\sca \nabla \phi\big\rangle \; ds+  \nu 
\int_{0}^{t}\big\langle \xi(s,\cdot),  \Delta \phi\big\rangle \; ds.
\end{equation}
\end{theorem}

\begin{remark}
Note that the limiting PDE \eqref{PDEintr2} does not depend on the value of the parameter $\beta \in (0, \mm{\frac{1}{4+2\alpha - \frac2p }})$.
\end{remark}

\begin{remark}
 The previous result implies, by the Sobolev embeddings' Theorem (see \cite[Section 2]{Triebel} for instance), the strong convergence in $C([0,T]\times K)$ for every 
compact set $K\subset\mathbb{R}^{2}$.
\end{remark}

Here it follows an outline of the paper: we start in Section \ref{sec:proof1} with an exposition of the strategy to prove Theorem \ref{Thm 1}. We will prove the technical estimates in Section \ref{sec:proof-tech}.  We chose to investigate in details the case where $\xi^{\mathrm{ini}}$ is
a probability density, in particular is non-negative and then, in Appendix \ref{app}, we
show that the same result holds in the general case $\xi^{\mathrm{ini}}%
\in\mathbb{L}^{1}(\mathbb{R}^{2})\cap\mathbb{L}^{\infty}(\mathbb{R}^{2})$,
without assumptions on the sign and the value of $\int\xi^{\mathrm{ini}%
}\left(  x\right)  dx$, by a simple decoupling argument. 
Finally, concerning the
importance and geometrical interpretation of the uniform convergence of the
mollified empirical measure, we have included a short discussion in Appendix
\ref{appB}. In particular, the uniform convergence  does not follow from the weak convergence of the
empirical measure.

\section{Strategy of the proof of Theorem \ref{Thm 1}}
\label{sec:proof1}

 There are \mm{three} main steps to derive the convergence result stated in Theorem \ref{Thm 1}:
\begin{enumerate}
\item  First, we write the \textit{mild formulation} of the identity satisfied by
\begin{equation}
g_{t}^{N}(x)=(V^{N}\ast S_{t}^{N})(x)=\int_{\RR^{2}} V^{N}\left(
x-y\right)  dS_{t}^{N}\left(  y\right),  \label{def g}%
\end{equation}
see Section \ref{sec:mild}. We will obtain a ``closed" inequality (note that $g^{N}$ already appears in the right hand side of \eqref{ident fro S^N}), and then prove two uniform bounds, see Proposition \ref{prop:bound1} and Proposition \ref{prop:bound2}.  To that aim 
we will use two main properties of the function $F$: first
 $F$ is Lipschitz continuous, and second it is bounded, precisely $| F(x)-F(y)|\leq |x-y |$ and   $\|F\|_{\mathbb{L}^\infty(\mathbb{R}^2)} \leq M$.

\item Then we apply compacteness arguments and Sobolev embeddings to have
subsequences which converge so as to pass to the limit, see Section \ref{subsect compactness}, Section \ref{sec:pass} and Section \ref{sec:conv}.

\item Finally, we are able to conclude the proof since the solution to the limiting PDE \eqref{PDEintr2} is unique, as is it proved in Section \ref{sec:unique}.

\end{enumerate}

The support of a function $f$ is denoted by $\mathrm{Supp} f$. When a constant $C$ will depend on some parameter, say $\alpha$, this will be highlighted in index by $C_\alpha$, but the constant may change from line to line. 

 \subsection{Analytic semigroup}

Let us first introduce the operator \[A:\mathcal{D}(
A)  \subset \mathbb L^{\mm{p}}(  \mathbb{R}^{d})  \rightarrow
\mathbb L^{\mm{p}}( \mathbb{R}^{d})  \] defined as $Af=\nu\Delta f$. It is the
infinitesimal generator of an analytic semigroup (the heat semigroup) in
$\mathbb L^{\mm{p}}(  \mathbb{R}^{d})  $ (see for instance \cite{Pazy}). We denote
this semigroup by $\{e^{tA}, \; t\geq0\}$, which is simply given by%
\[
\left(  e^{tA}f\right)  \left(  x\right)  =\int_{\mathbb{R}^{d}}\frac
{1}{\left(  4\nu\pi t\right)  ^{d/2}}e^{-\vert x-y\vert ^{2}/(4\nu t)}f\left(  y\right)  dy, \qquad f \in \mathbb{L}^\mm{p}(\RR^d).
\]
Moreover, denoting by $\mathrm{I}$ the identity operator, we know that, for any $\varepsilon \in \RR$, the domain of the operator $(\mathrm{I}-A)^{\varepsilon/2}$ is given by%
\[
\mathcal{D}\big(  \left(  \mathrm{I}-A\right)  ^{\varepsilon/2}\big)  = \mathbb H_{\mm{p}}^{\varepsilon}(
\mathbb{R}^{d})
\]
with equivalent norms, where    $\left(\mathrm{I}-A\right)^{\varepsilon/2}$ is the Bessel potential operator given by $\left(\mathrm{I}-A\right)^{\varepsilon/2}f= \mathcal{F}^{-1} ((1+|\cdot|^{2})^{\varepsilon/2}\;\mathcal{F} f(\cdot)) $. Recall also from \cite{Pazy} that, for every $\varepsilon>0$,
and $T>0$, \mm{and $p>1$}, there is a constant $C_{\varepsilon,T,\nu,\mm{p}} >0$ 
such that, for any $t \in (0,T]$,
\begin{equation}
\big\Vert\left(  \mathrm{I}-A\right)^{\varepsilon}e^{tA}\big\Vert_{\mathbb L^{\mm{p}}\rightarrow \mathbb L^{\mm{p}}}\leq\frac{C_{\varepsilon,T,\nu,\mm{p}}}{t^{\varepsilon}}.%
\label{eq:etoile}\end{equation}
We are now ready to prove Theorem \ref{Thm 1}, therefore the dimension is from now on $d=2$.

\subsection{The equation for $V^{N}\ast S_{t}^{N}$ in
mild form}
\label{sec:mild}
We want to
deduce an identity for $g_{t}^{N}(x)$ from \eqref{ident fro S^N}.
For every  $x\in\mathbb{R}^{2}$ take, in identity \eqref{ident fro S^N}, the test
function $\phi_{x}\left(  y\right) =V^{N}\left(  x-y\right)  $. 
We get (recall the definition \eqref{def g} of $g_{t}^{N}$)
\begin{align}
g_{t}^{N}  (x)    = 
g_{0}^{N}  (x)& +\int_{0}^{t}\left\langle S_{s}%
^{N}, F(K\ast g_{s}^{N}) \;\sca \; \nabla  V^{N}  \left(  x-\cdot\right)  \right\rangle ds \notag\\
& \label{equagn}
  +  \nu \int_{0}^{t}  \Delta g_{s}^{N} (x)ds+\frac{\sqrt{2\nu}}{N}\sum_{i=1}^{N}\int_{0}^{t}\nabla  V^{N}  \left(  x-X_{s}^{i,N}\right) \sca\;
dW_{s}^{i}.
\end{align}
In the sequel, let us write for the sake of clarity,%
\[
\big\langle S_{s}^{N},F(K\ast g_{s}^{N}) \; \sca \;
\nabla V^{N} \left(  x-\cdot\right)  \big\rangle =:\Big(
\nabla V^{N} \ast\left(
F(K\ast g_{s}^{N}) S_{s}^{N}\right)  \Big)  \left(  x\right)
\]
and similarly for analogous expressions. Hence we can write (using the same idea as in   \cite{flan}):
\begin{align*}
g_{t}^{N}    =e^{tA} {g}_{0}^{N} & +\int_{0}^{t}e^{\left(  t-s\right)  A%
}\left(  \nabla  V^{N}\ast\left(
F(K\ast g_{s}^{N}) S_{s}^{N}\right)  \right)  ds\\
& +\frac{{\sqrt{2\nu}}}{N}\sum_{i=1}^{N}\int_{0}^{t}e^{\left(  t-s\right)  A}\left(
\nabla V^{N}  \left(
\cdot-X_{s}^{i,N}\right)  \right)  \sca\; dW_{s}^{i}.
\end{align*}
By inspection in the convolution explicit formula for $e^{\left(  t-s\right)
A}$, one can see that \[e^{\left(  t-s\right)  A}\nabla f=\nabla e^{\left(
t-s\right)  A}f,\] and then one can use the semigroup property, so as to write
\begin{align}
g_{t}^{N}   =e^{tA} {g}_{0}^{N}  &+\int_{0}^{t}   \nabla e^{\left(  t-s\right)  A%
}\left(    V^{N}\ast\left(
F(K\ast g_{s}^{N}) S_{s}^{N}\right)  \right)  ds \notag
\\ \label{equatiog1h}
 & +\frac{{\sqrt{2\nu}}}{N}\sum_{i=1}^{N}\int_{0}^{t}e^{\left(  t-s\right)  A} \left(
\nabla V^{N} \left(\cdot-X_{s}^{i,N}\right)\right) \sca \;  dW_{s}^{i}.
\end{align}

\mm{Recall that three parameters $p\geqslant 2$, $\alpha \in (\frac2p,1)$ and $\beta$ are fixed  from Assumption \ref{assump} for the rest of the paper. From now on every constant $C_\lambda$ which depends on some parameter $\lambda$ may also depend on the three parameters $\alpha,p,\beta$: we decide to withdraw them from the notation in order not to burden the paper.} In the following we  will prove two important bounds: 

\begin{proposition}\label{prop:bound1} We assume Assumption  \ref{assump}. \mm{Let $q\geqslant 2$}. 
Then there exists a positive constant $C_{\mm{T,M, \nu,q}}$ such that, for
all $t\in(0,T]$ and $N\in\mathbb{N}$, it holds:
\begin{equation}
\mathbb{E}\left[  \left\Vert \left(  \mathrm{I}-A\right) ^{\mm{\alpha}/2} g_{t}^{N}   \right\Vert _{\mm{\mathbb L^{p}}\left(  \mathbb{R}^{2}\right)  }^{\mm{q}}\right]  \leq
C_{\mm{T, M, \nu, q}}. \label{eq:firstestim}
\end{equation}\end{proposition}

\begin{proposition}\label{prop:bound2}We assume Assumption  \ref{assump}. Let $\gamma  \in(0,\frac{1}{2})$ and $\mm{q'}\geqslant 2$.  There exists a positive constant $C_{T,M,\nu,\mm{q'}}$ such that, for
any $N\in\mathbb{N}$, it holds:
\begin{equation}
\EE\Bigg[ \int_{0}^{T}\int_{0}^{T}  \frac{ \big\|g_t^{N}-g_s^{N}\big\|_{-2,2}^{\mm{q'}}}{|t-s|^{1+ \mm{q'}\gamma}}      \; ds \; dt \Bigg]  \leq C_{T, M, \nu, \mm{q'}}.
\label{eq:secondestim} \end{equation}
\end{proposition}

The proofs of Proposition \ref{prop:bound1} and Proposition \ref{prop:bound2} are postponed to Section \ref{ssec:proof1} and Section \ref{ssec:proof2}, respectively.

\subsection{Criterion of compactness} 

\label{subsect compactness}

In this  subsection we follow  the arguments of  \cite[Section 3.1]{flan3}.  We start by constructing a space on which the sequence of the probability laws of $g_\cdot^N$ is tight. 

\bigskip

We exploit Corollary 9 of  Simon \cite{Simon}, using as far as possible the notations of that paper. Given 
a ball $\mathcal{B}_R \subset \mathbb{R}^{2}$, taken $  \mm{\alpha>\frac2p}$ (as in Theorem \ref{Thm 1}), and $\mm{\frac2p}<\eta<\alpha$, we consider the space 
\[
X:= \mathbb{W}^{\alpha, \mm{p}}(\mathcal{B}_R), \qquad  B:= \mathbb{W}^{\eta, \mm{p}}(\mathcal{B}_R),\qquad Y:= \mathbb{W}^{-2, 2}(\mathcal{B}_R).
\]
One can check that $X\subset B\subset Y $,
with compact dense embedding.  Moreover, we have the interpolation inequality: for any    $f\in X $
\[
\| f\|_{B}\leq C_{R} \;  \| f \|_{X}^{1-\theta} \;  \| f \|_{Y}^{\theta}
\]
with \begin{equation}
\theta:= \frac{\alpha-\eta}{2+\alpha}.  \label{eq:theta}
\end{equation}
Now taking, in the notations of \cite[Corollary 9]{Simon}: $s_0:=0$, $s_{1}:=\gamma\in(0,\frac12)$, and choosing $\mm{q,q'}\geqslant 2$ such that : 
 \begin{equation}s_1\mm{q'}=\gamma \mm{q'} >1, \qquad \text{and} \qquad s_{\theta}:=\theta s_{1} = \theta \gamma > \frac{1-\theta}{\mm{q}}+ \frac{\theta}{\mm{q'}}, \label{eq:condition}\end{equation} then
the corollary tells us  that   the space $\LL^{\mm{q}}([0, T]\; ;\; X)\cap \mathbb{W}^{\gamma, \mm{q'}}([0,T] \; ; \; Y) $
is relatively compact in $C([0,T] \; ; \; B)$.

 Therefore, for any $\gamma \in (0,\frac12)$, for the parameter $\alpha>\mm{\frac2p}$ given by point \textit{3.} of Assumption \ref{assump}, and for $\mm{q,q'} \geqslant 2$ which satisfy \eqref{eq:condition},  we now consider the space
\[
\mathfrak{Y}_{0}:=
\LL^{\mm{q}}\big(  [0,T]\; ;\; \HH_{\mm{p}}^{\alpha} \big)  \cap
\mathbb{W}^{\gamma,\mm{q'}}\big(  [0,T]\; ; \; \HH_{2}^{-2} \big)  .
\]
We use the Fr\'{e}chet topology on $C\big( [ 0,T]\; ;\; \mathbb{W}_{\rm loc}^{\eta,\mm{p}}(  \mathbb{R}^{2} ) \big)  $ defined as
\[
d\left(  f,g\right)  =\sum_{n=1}^{\infty}2^{-n}\bigg(  1\wedge  \sup_{t\in [0,T]} \big\Vert (f-g) (  t,\cdot )  \big\Vert _{\mathbb{W}^{\eta
,\mm{p}} (  \mathcal{ B}_n   )  }^{2}\bigg),
\]
where $\wedge$ denotes the infimum. From the above, we conclude that $\mathfrak{Y}_0$ is compactly embedded into  $C\big( [ 0,T]\; ;\; \mathbb{W}_{\rm loc}^{\eta,\mm{p}}  \big)$ for any $\mm{\frac2p}<\eta< \alpha$. 
Finally, let us denote  by $\LL_{w}^{2}$ the  spaces
  $\LL^{2}$ endowed with the weak topology. We obtain that
$\mathfrak{Y}_{0}$ is compactly embedded into%
\begin{equation}
\mathfrak{Y}:= \LL_{w}^{2}\big( [ 0,T]\; ;\; \HH_{\mm{p}}^{\alpha}  \big)
\cap C(  [0,T]\; ;\; \mathbb{W}_{\rm loc}^{\eta,\mm{p}} \big). \label{topology of convergence}%
\end{equation}
Note that%
\[
C\big( [ 0,T]\; ;\; \mathbb{W}_{\rm loc}^{\eta,\mm{p}}
\big)  \subset C(  [0,T]\; ;\; C(  D)  \big)
\]
for every regular bounded domain $D\subset\mathbb{R}^{2}$.

\bigskip

 Let us now go back to the sequence of processes $\{g_\cdot^N\}_N$, for which we have proved several estimates. The  Chebyshev inequality ensures that 
\[
\PP\big( \|g_{\cdot}^{N} \|_{\mathfrak{Y}_{0}}^{2} > R\big)\leq \frac{\EE \big[ \big\| g_{\cdot}^{N}\big\|_{\mathfrak{Y}_{0}}^{2}\big]}{R}, \qquad \text{for any } R>0.
\]
Thus by Proposition \ref{prop:bound1} and Proposition \ref{prop:bound2} (since $\mm{q,q'} \geqslant 2$), we obtain 
\[
\PP\big( \big\|g_{\cdot}^{N} \big\|_{\mathfrak{Y}_{0}}^{2} > R\big)\leq \frac{C}{R}, \qquad \text{for any } R>0, N \in \NN.
\]
The process $t \in [0,T] \mapsto g_{t}^{N}$ defines a probability $\mm{\mathbf{P}}_{N}$ on $\mathfrak{Y}$. Fix $\varepsilon>0$. Last inequality implies that there exists a bounded set $B_{\mm{\varepsilon}}\in \mathfrak{Y}_{0}$ 
such that $\mm{\mathbf{P}}_{N}(B_{\mm{\varepsilon}})< 1-\mm{\varepsilon}$ for all $N$, and therefore, from the previous argument, there exists a compact set $
K_{\mm{\varepsilon}}\subset \mathfrak{Y}$ such  that $\mm{\mathbf{P}}_{N}(K_{\mm{\varepsilon}})< 1-\mm{\varepsilon}$.

\bigskip

Finally, denote by $\{  \texttt{L}^{N}\}_{N\in\mathbb{N}}$ the laws of the processes $\{
g^{N}\}  _{N\in\mathbb{N}}$ on $\mathfrak{Y}_{0}$,  we have proved that $\{
\texttt{L}^{N}\}  _{N\in\mathbb{N}}$ is tight in $\mathfrak{Y}$, hence relatively compact,
by Prohorov's Theorem. From every subsequence of $\{  \texttt{L}^{N}\}_{N\in\mathbb{N}}$ it is possible to extract a further subsequence which
converges to a probability measure $\texttt{L}$ on $\mathfrak{Y}$. Moreover  by a theorem of
Skorokhod (see \cite[Theorem 2.7]{Ikeda}), we are allowed, eventually after choosing a suitable probability space  where all our random variables can be defined, to assume
\begin{equation} \label{eq:lim}
g^{N} \rightarrow \ \xi \quad \text{in }   \mathfrak{Y}, \qquad \text{a.s.}
\end{equation}
where the law of $\xi$ is $\texttt{L}$.

\subsection{Passing to the limit}
\label{sec:pass}

Next step is to characterize the limit. 
First, recall formula  \eqref{equagn}, which reads
\begin{align*}
g_{t}^{N}  (x)    =
g_{0}^{N}  (x)&+\int_{0}^{t}\left\langle S_{s}%
^{N},F(K\ast g_{s}^{N}) \sca \; \nabla  V^{N}  \left(  x-\cdot\right)  \right\rangle ds\\
&   + \nu \int_{0}^{t}  \Delta g_{s}^{N} (x)ds+\frac{{\sqrt{2\nu}}}{N}\sum_{i=1}^{N}\int_{0}^{t}\nabla  V^{N}  \left(  x-X_{s}^{i,N}\right) \sca \;
dW_{s}^{i}.
\end{align*}
Taking a  test function $\phi: \RR^2 \to \RR$ we have 
\begin{align}
\left\langle g_{t}^{N}, \phi  \right\rangle   =
\left\langle  g_{0}^{N},  \phi \right\rangle &+\int_{0}^{t}\left\langle S_{s}%
^{N},F(K\ast g_{s}^{N}) \sca \; \nabla  (V^{N}\ast \phi)  \right\rangle ds\notag\\
&  + \nu \int_{0}^{t}  \left\langle g_{s}^{N}, \Delta \phi \right\rangle ds+
\frac{{\sqrt{2\nu}}}{N}\sum_{i=1}^{N}\int_{0}^{t}\nabla ( V^{N}\ast \phi) \left(X_{s}^{i,N}\right) \sca \;
dW_{s}^{i}. \label{eq:arg1}
\end{align}
It is clear from \eqref{eq:lim} that 
\begin{equation}\label{eq:arg2}
\left\langle g_{t}^{N}, \phi  \right\rangle  \xrightarrow[N\to\infty]{} \left\langle \xi  , \phi  \right\rangle , \qquad
\left\langle  g_{0}^{N},  \phi \right\rangle \xrightarrow[N\to\infty]{} \left\langle  \xi^{\rm ini},  \phi \right\rangle ,
\end{equation}
and  also that
\begin{align} \notag
\EE\Bigg[  \bigg|\frac{1}{N}\sum_{i=1}^{N}\int_{0}^{t}\nabla ( V^{N}\ast \phi) \left(X_{s}^{i,N}\right) \sca\;
dW_s^i\bigg|^{2}\bigg]
&=\frac{1}{N^{2}}\sum_{i=1}^{N}\int_{0}^{t}  \EE\Big[\big|\nabla ( V^{N}\ast \phi) \left(X_{s}^{i,N}\right)\big|^{2}\Big] 
ds \\
&\leq\frac{t}{N}  \|\nabla\phi\|_{\mathbb{L}^\infty}^{2}  \xrightarrow[N\to\infty]{} 0.\label{eq:arg3}
\end{align}
We now claim that 
\begin{equation}\label{eq:claim}
\lim_{N\rightarrow \infty}\int_{0}^{t}\big\langle S_{s}%
^{N}, F(K\ast g_{s}^{N})\sca\;\nabla  (V^{N}\ast\phi) \big\rangle \; ds= \int_{0}^{t} \int_{\RR^2} \xi(s,x) F(K\ast {\xi}) \sca \;\nabla \phi(x) \; dx  ds.
\end{equation}

\noindent \textbf{Proof of \eqref{eq:claim}.} 
First, we observe  that   $\big\|g_{t}^{N}\big\|_{\LL^1}=1$
 and  $g_{\cdot}^{N}$   is uniformly bounded  in $ \LL^{2}(  [0,T]\; ;\; \HH_{\mm{p}}^{\eta}) $ for any $\eta> \mm{\frac2p}$. 
Then by the Sobolev embeddings' Theorem (see \cite[Section 2.8.1]{Triebel}), we have that $g_{\cdot}^{N}$   is uniformly bounded in $\LL^{2}\big(  [0,T]\; ;\; C_{b}( \mathbb{R}) \big) $.

\bigskip

 By interpolation we also know that $g_{\cdot}^{N}\in \LL^{2}([0,T]\; ;\;  \LL^{\mm{a}}(\RR^{2}))$ for any $\mm{a} \in [1,+\infty]$, and there is a constant $C_{K}>0$ such that
\begin{equation}\label{k_{1}}
\big\|K\ast g_{t}^{N}\big\|_{\LL^\mm{a}}\leq C_{K}  \big\| g_{t}^{N}\big\|_{\LL^\mm{b}}
\end{equation}
with $\frac{1}{\mm{a}}= \frac{1}{\mm{b}}-\frac{1}{2}$ and $1< \mm{b}< 2$. 
The kernel $K$ is singular, of Calderon-Zygmund type. Hence, there exists
a constant $C>0$ such that, for any $\mm{a}\in(1,+\infty)$, 
\begin{equation}\label{k_{2}}
\big\| \nabla (K\ast g_{t}^{N})\big\|_{\LL^\mm{a}}\leq  C\big\| g_{t}^{N} \big\|_{\LL^\mm{a}}.
\end{equation}  
 Let us introduce $f_\cdot^{N}=K\ast g_{\cdot}^{N}$. 
By the Sobolev embeddings' Theorem,  estimates \eqref{k_{1}} and 
\eqref{k_{2}} imply  that: for any $\tilde\eta >0$ (take $\tilde\eta= 1- \frac{2}{\mm{a}}$  with $\mm{a}>2$),
 \[
\| f^{N}\big\|_{\LL^{2}([0,T] \; ; \; C^{\tilde\eta}(\RR^{2}))}\leq C. 
 \]
Now, let $\chi: \RR^2 \to [0,1] \in C_{0}^{\infty}$ be a \textit{cut-off function}, \textit{i.e.}~such that: \[0\leq \chi(x) \leq 1, \qquad \text{and} \qquad \chi(x)=1, \text{ if } |x|\leq 1.\] Let $\tilde{\chi}= 1- \chi$. 
We can decompose:
\begin{multline*}
\left( K\ast \big(g_{t}^{N}-\xi\big)\right)(x)\\=\int_{\RR^{2}} \chi(y) K(y)  (g_{t}^{N}-\xi)(x-y) dy + \int_{\RR^{2}} \tilde\chi(x-y) K(x-y)  (g_{t}^{N}-\xi)(y) dy.
\end{multline*}
We observe that $\chi(\cdot) K(\cdot)\in \LL^{\mm{c}}(\RR^2)$ with $\mm{c}<2$, and  $\tilde{\chi}(x-\cdot)K(x-\cdot)\in \LL^{\mm{d}}(\RR^2)$
with $\mm{d}>2$. Since $g_{\cdot}^{N}$ is uniformly bounded in   $\LL^{2}([0,T]\; ; \;\LL^{\mm{a}}(\RR^{2}))$ for all $\mm{a} \in [1,+ \infty]$ 
we  obtain that  $K\ast g_{t}^{N}$ converges to $K\ast \xi$. 

\bigskip

Finally, we can bound as follows:
\begin{multline*}
\Big| \Big\langle S_{s}^{N},F(K\ast g_{s}^{N})\sca\;\nabla  (V^{N}\ast\phi)  \Big\rangle
-\Big\langle g_{s}^{N},F(K\ast g_{s}^{N})\sca\;\nabla  (V^{N}\ast\phi)   \Big\rangle \Big|
\\
\leq \sup_{x\in \RR^{2} }\Big| F(K\ast g_{s}^{N})\sca\;\nabla  (V^{N}\ast\phi)(x)  
- \big( F(K\ast g_{s}^{N})\sca\;\nabla  (V^{N}\ast\phi) \big) \ast  V^{N} \big)(x) \Big|.
\end{multline*}
 Let us control the last term, using the facts that \begin{itemize} \item $V$ is a density (denoted below by $(\int V =1)$), 
 \item  $F$ is Lipschitz and bounded (denoted below by ($F\in$ Lip$\cap L^\infty$)), \item $V$ is compactly supported (denoted below by ($V$ is c.s.)),
 \item and $\phi$ is compactly supported and smooth, 
  \end{itemize}  as follows (the norm  $\|\cdot\|$ below is the Euclidean norm on $\RR^2$): for any $x\in\RR^2$,
\begin{align*}
  \Big| & F(K\ast g_{s}^{N})(x) \sca\;\nabla (V^{N}\ast\phi)(x)  
- \big( F(K \ast g_{s}^{N})\sca\;\nabla (V^{N}\ast\phi) \big) \ast V^{N} \big)(x) \Big| \vphantom{\Bigg\{}\\
&
 \overset{\substack{\hphantom{(F\in\mathrm{Lip}\cap L^\infty)}\\(\int V=1)}}{\leq}   \int_{\RR^2} V(y) \; \Big\|\nabla (V^{N}\ast\phi)(x)\Big\| \;
 \Big\|   F(K\ast g_{s}^{N})(x)- F(K\ast g_{s}^{N})\big(x -\tfrac{y}{N^{\beta}}\big) \Big\| dy \notag \\
 & \quad\; \qquad
 +    \int_{\RR^2} V(y)  \ \Big\|\nabla (V^{N}\ast\phi)(x)-\nabla(V^{N}\ast\phi)\big(x- \tfrac{y}{N^{\beta}} \big)  \Big\| \; 
\big\|(K\ast g_{s}^{N})(x)\big\|dy 
 \\ &
\overset{(F\in\mathrm{Lip}\cap L^\infty)}{\leq}  C \int_{\RR^2} V(y) \; \big\|\nabla (V^{N}\ast\phi)(x)\big\| \;
 \Big\| f_{s}^{N}(x)-f_{s}^{N}\big(x-\tfrac{y}{N^{\beta}} \big)\Big\| dy \\ 
 & \qquad \qquad 
 +  \frac{C}{N^{\beta}}  \int_{\RR^2} V(y) \|y\|     dy \\
& 
 \overset{\substack{\hphantom{(F\in\mathrm{Lip}\cap L^\infty)}\\(V \mathrm{\; is  \; c.s.})}}{\leq}   \frac{C}{N^{\tilde\eta \beta}}  \ \sup_{x,y\in \mathbf{K}} \frac{\big\| f_{s}^{N}(x)-f_{s}^{N}(y)\big\|}{\|x-y\|^{\tilde\eta}} \ \int_{\RR^2} V(y)  \|y\|^{\tilde\eta} \;  dy \\ 
 & \qquad  \qquad +   \frac{C}{N^{\beta}} \int_{\RR^2} V(y) \|y\|     dy ,
\end{align*}
where $\mathbf{K}\subset \RR^2$ is a compact set. 
Therefore we have obtained 
\[
  \Big|  F(K\ast g_{s}^{N})(x) \sca\;\nabla (V^{N}\ast\phi)(x)  
- \big( F(K \ast g_{s}^{N})\sca\;\nabla (V^{N}\ast\phi) \big) \ast V^{N} \big)(x) \Big|
\leq  \frac{C} {N^{\tilde\eta \beta}} ,
\]
where the constant $C$ depend on $\| f^{N}\|_{C^{\tilde\eta}(\RR^{2})}$. Thus,
\begin{align*}
\lim_{N\rightarrow \infty}\int_{0}^{t}  \Big\langle S_{s}%
^{N},F(K\ast g_{s}^{N}) \; & \sca\; \nabla  (V^{N}\ast\phi)   \Big\rangle \; ds \\
& =\lim_{N\rightarrow \infty}\int_{0}^{t}\Big\langle g_{s}%
^{N},F(K\ast g_{s}^{N}) \sca\; \nabla  (V^{N}\ast\phi)  \Big\rangle \; ds.
\\ & =
\lim_{N\rightarrow \infty}\int_{0}^{t}   \int_{\RR^2}  g_{s}^{N}(x)  F(K\ast g_{s}^{N})(x) \sca\; \nabla (V^{N}\ast\phi)\left(x\right)  \; dx ds \\
& =  \int_{0}^{t}  \int_{\RR^2}  \xi(s,x)  F(K\ast \xi(s,\cdot))(x) \sca\; \nabla \phi(x)  \; dx ds
\end{align*}
 where in the last equality we used that $g_s^{N}\rightarrow \xi$ strongly in 
$\LL^{2}\left( [ 0,T]\; ;\; C(  \RR^2)  \right)$. We have proved \eqref{eq:claim}. \qed

\bigskip

Therefore from \eqref{eq:arg1}, \eqref{eq:arg2}, \eqref{eq:arg3} and \eqref{eq:claim}, we conclude that the limit point $\xi$ is then a  weak solution  to the PDE 
\begin{equation}
\partial_{t}\xi+  \mathrm{div} \big(\xi \; F(K\ast \xi) \big) = \nu \Delta \xi, \qquad
\xi\big|_{t=0}= \xi^{\rm ini}. \label{PDEintrAUX}
\end{equation}
Note that there is one more step to recover \eqref{Vorty2} in Theorem \ref{Thm 1}, which will be achieved in Section \ref{sec:conclusion} below. Before that, we need to prove that the solution to \eqref{PDEintr2} is unique.

\subsection{Uniqueness of the solution}
\label{sec:unique}

\begin{theorem}
\label{Thm PDE} We assume that $\xi^{\rm ini}\in \LL^1 \cap \LL^{\infty}(\RR^2)$. Then there is at most one weak solution of the equation \eqref{PDEintr2}  which belongs to  $\LL^{2}\big( [ 0,T]\; ;\; \LL^{1}\cap \LL^\infty(\RR^{2})\big)$.  
\end{theorem}

\begin{proof} For any function $u:[0,T] \times \RR^2 \to \RR$ we introduce the notation \[\left\| u  \right\|_{\LL^{1}\cap \LL^{\infty}} := \| u(t,\cdot) \|_{\LL^1} + \| u(t,\cdot) \|_{\LL^\infty} .\] 
Let $\xi^{1},\xi^{2}$ be two weak solutions of  \eqref{PDEintr2},
 with the same initial condition $\xi^{\rm ini}$ which satisfies $\xi \in \LL^1 \cap \LL^{\infty}(\RR^2)$.
By hypothesis,  from \cite{Ben}, $\xi^{1}, \xi^{2} \in \LL^{2}(  [0,T]\; ;\; \LL^\infty) $.

 Let $\{h_{\varepsilon}\}_{\varepsilon}$ be a family of standard symmetric mollifiers  on $\RR^2$. For any
$\varepsilon>0$ and $x\in\mathbb{R}^{2}$ we can use $h_{\varepsilon
}(x-\cdot)$ as a test function in the equation \eqref{limiteq}.  Therefore we set
$\xi_{\varepsilon}^{i}(t,x)=(\xi^{i}(t,\cdot)\ast h_{\varepsilon})(x)$ for
$i=1,2$. Then we have, for any $(t,x) \in [0,T] \times \RR^2$, %
\[
\xi_{\varepsilon}^{i}(t,x)=(\xi^{\rm ini}\ast h_{\varepsilon})(x)+  \nu \int_{0}^{t} \Delta
\xi_{\varepsilon}^{i}(s,x)\,ds+\int_{0}^{t}\big((\nabla h_{\varepsilon}\sca \; F(K\ast \xi^{i})) \ast \xi^{i}\big)(s,x)  \,ds.
\]
Writing this identity in mild form we obtain (writing with a little abuse of notation $\xi^{i}\left(  t\right)
$ for the function $\xi^{i}\left(  t,\cdot\right)  $ and $\mathfrak{S}(t)$ for $e^{tA}$): 
\[
\xi_{\varepsilon}^{i}(t)=\mathfrak{S}(t)\sca\; (\xi^{\rm ini}\ast h_{\varepsilon})+\int_{0}%
^{t} \mathfrak{S}(t-s)\sca \; \big((\nabla h_{\varepsilon}\sca \; F(K\ast \xi^{i})) \ast \xi^{i} \big)(s)  ds.
\]
The function $X=\xi^{1}-\xi^{2}$ satisfies 
\[
h_{\varepsilon}\ast X(t)=\int_{0}^{t} \nabla \mathfrak{S}(t-s) \sca  \Big(  h_{\varepsilon}\ast \Big( F(K\ast \xi^{1}) \xi^{1}- F(K\ast \xi^{2}) \xi^{2}\Big)\Big) \,ds.
\]
Thus we obtain 
\[
\|h_{\varepsilon}\ast X(t)\|_{\LL^{\infty}}\leq 
\int_{0}^{t} \Big\|  \nabla \mathfrak{S}(t-s) \sca \Big( h_{\varepsilon}\ast \Big( F(K\ast \xi^{1}) \xi^{1}- F(K\ast \xi^{2}) \xi^{2}\Big) \Big) \Big\|_{\LL^{\infty}} \,ds.
\]
Therefore, there is a constant $C_\nu >0$ such that
\[
\|h_{\varepsilon}\ast X(t)\|_{\LL^{\infty}}\leq 
\int_{0}^{t}  \frac{{C_\nu}}{(t-s)^{\frac{1}{2}}} \; \Big\|   h_{\varepsilon}\ast \Big( F(K\ast \xi^{1}) \xi^{1}- F(K\ast \xi^{2}) \xi^{2}\Big)  \Big\|_{\LL^{\infty}} \,ds.
\]
Taking the limit as $\varepsilon\rightarrow0$ we  arrive  at
\[
\| X(t)\|_{\LL^{\infty}}\leq 
\int_{0}^{t}  \frac{{C_\nu}}{(t-s)^{\frac{1}{2}}}  \; \Big\|   F(K\ast \xi^{1}) \xi^{1}- F(K\ast \xi^{2})\xi^{2} \Big\|_{\LL^{\infty}} \,ds.
\]
With similar arguments we have  the same estimate in $\LL^1$--norm as follows:
\[
\| X(t)\|_{\LL^{1}}\leq 
\int_{0}^{t}  \frac{{C'_\nu}}{(t-s)^{\frac{1}{2}}} \;  \Big\|   F(K\ast \xi^{1}) \xi^{1}- F(K\ast \xi^{2}) \xi^{2} \Big\|_{\LL^{1}} \,ds.
\]
By easy calculation we have 
\begin{align*}
\| X(t)\|_{\LL^{\infty}} & \leq 
\int_{0}^{t}  \frac{{C_\nu}}{(t-s)^{\frac{1}{2}}}  \Big( \big\|  X  F(K\ast \xi^{1})  \big\|_{\LL^{\infty}} + \big\|  \xi^{2} \big(F(K\ast \xi^{1})- F(K\ast \xi^{2})\big) \big\|_{\LL^{\infty}} \Big)\,ds.
\\ & 
 \leq {C_\nu} \int_{0}^{t}  \frac{\left\| \xi^{1} \right\|_{\LL^{1}\cap \LL^{\infty}}}{(t-s)^{\frac{1}{2}}}  \big\|  X \big\|_{\LL^{\infty}}  \; ds   +   {C_\nu}\int_{0}^{t}  \frac{\left\| \xi^{2} \right\|_{\LL^{1}\cap \LL^{\infty}}}{(t-s)^{\frac{1}{2}}}  \big(  \big\| X  \big\|_{\LL^{\infty}}  +     \big\| X \big\|_{\LL^{1}} \big) ds 
\\ & 
\leq  {C_\nu}  \int_{0}^{t}  \frac{\left\| \xi^{{1}} \right\|_{\LL^{1}\cap \LL^{\infty}}+ \left\| \xi^{2} \right\|_{\LL^{1}\cap \LL^{\infty}}}{(t-s)^{\frac{1}{2}}} \ \big( \big\| X  \big\|_{\LL^{\infty}}  + \big\| X \big\|_{\LL^{1}} \big) ds.  
\end{align*}
On the other hand, in a similar way we have 
\begin{align*}
\| X(t)\|_{\LL^{1}} & \leq  {C'_\nu}
 \int_{0}^{t}  \frac{1}{(t-s)^{\frac{1}{2}}} \Big(  \big\|  X  F(K\ast \xi^{1})  \big\|_{\LL^{1}}  +  \big\|  \xi^{2} \big(F(K\ast \xi^{1})-F(K\ast \xi^{2}) \big) \big\|_{\LL^{1}} \Big) \,ds.
\\ 
 & 
\leq   {C'_\nu} \int_{0}^{t}  \frac{\left\| \xi^{{1}} \right\|_{\LL^{1}\cap \LL^{\infty}}+ \left\| \xi^{2} \right\|_{\LL^{1}\cap \LL^{\infty}}}{(t-s)^{\frac{1}{2}}} \ \big( \big\| X \big\|_{\LL^{\infty}}  + \big\| X \big\|_{\LL^{1}}  \big) ds. 
\end{align*} 
  Therefore we have, for a constant $C''_\nu >0$, that 
\[
\| X (t)\|_{\LL^{1}\cap \LL^{\infty}}\leq {C''_\nu}  \int_{0}^{t}  \frac{\left\| \xi^{{1}} \right\|_{\LL^{1}\cap \LL^{\infty}}+ \left\| \xi^{2} \right\|_{\LL^{1}\cap \LL^{\infty}}}{(t-s)^{\frac{1}{2}}} \;  \| X(s) \|_{\LL^{1}\cap \LL^{\infty}} \; ds. 
\]
By Gronwall's Lemma we conclude $X=0$.
\end{proof}

\subsection{Convergence in probability}

\label{sec:conv}

\begin{corollary}\label{coroconveP2} The sequence $\{g^{N}\}_{N\in\NN}$ converges in probability to $\xi$.
\end{corollary}

\begin{proof} 
 We denote  the joint law of $ (g^{N} ,g^{M} )$  by $\nu^{N,M}$. Similarly  to the proof of tightness for $g^{N}$ (Section \ref{subsect compactness}) we have that the family $\{\nu^{N,M}\}$ is tight in    $\mathfrak{Y} \times  \mathfrak{Y}$, where $\mathfrak{Y}$ has been defined in \eqref{topology of convergence}.

Let us take any subsequence $ \nu^{N_k,M_k} $. By  Prohorov's  Theorem, it is
relatively weakly compact hence it contains a weakly convergent subsequence. Without
loss of generality we may assume that the original sequence $\{\nu^{N,M} \}$ itself converges
weakly to a measure $\nu$. According to  Skorokhod immersion's Theorem, we
infer the existence of a probability space  $\big(  \widebar{\Omega}, \widebar{\mathcal{F}},\widebar{\mathbb{P}} \big)$   with a sequence of random variables
$(\widebar{g}^{N}, \widebar{g}^{M})$ converging almost surely in $\mathfrak{Y} \times \mathfrak{Y} $ to random variable  $   (\bar{u}, \check{u})$ and
 the laws of $(\widebar{g}^{N}, \widebar{g}^{M}) $  and $ (\bar{u}, \check{u})$
  under $ \widebar{\mathbb{P}}$  coincide with  $  \nu^{N,M} $ and   $\nu$, respectively.

	Analogously, it can be applied to both  $\widebar{g}^{N}$   and $ \widebar{g}^{M}$  in order
to show that  $\bar{u}$ and $\check{u}$    are  two solutions of the PDE \eqref{PDEintr2}. By Theorem \ref{Thm PDE} which gives the uniqueness of the solution we have  $\bar{u}=\check{u} $. 
 Therefore 
	$$
	\nu  \big((x,y)\in \mathfrak{Y}\times \mathfrak{Y} \; ; \; x=y\big)= \widebar{\mathbb{P}} (\bar{u}=\check{u})=1. 
$$
Now, we have all in hands to apply  Gyongy-Krylov's characterization
of convergence in probability, which reads as follows: 
\begin{lemma}[Gyongy-Krylov \cite{gyon}] \label{GK} Let $\{X_{n}\}$ be a sequence of random elements in a Polish
space $\Psi$ equipped with the Borel $\sigma$-algebra. Then $X_{n}$ converges in
probability to an $\Psi$-valued random element if, and only if, for each pair
$(X_{\ell},X_{m})$ of subsequences, there exists a subsequence $\{v_{k}\}$ given by
\[
v_{k}= (X_{\ell(k)}, X_{m(k)}),
\]
converging weakly to a random element $v(x,y)$ supported on the diagonal set
\[
\{ (x,y) \in\Psi\times\Psi: x= y \}.
\]

\end{lemma}

This lemma implies that the original
sequence is  defined on the initial probability space converges in probability
in the topology of $\mathfrak{Y}$ to a random variable $\mu$.
\end{proof}

\subsection{Conclusion}
\label{sec:conclusion}

Let $\xi(t,x)$ be the unique solution of the vorticity equation  \eqref{PDEintr2} with 
initial condition $ \xi^{\rm ini}\in \LL^{\infty}\cap \LL^{1}(\RR^{2})$. From \cite{Ben} we have $\| \xi\|_{\infty}\leq  \| \xi^{\rm ini}\|_{\infty}$. Then, by definition of the Biot-Savart kernel $K$, there is a positive constant $c_K$ (given by \eqref{eq:ck}) such that 
\[\| K \ast  \xi\|_{\infty}\leq  c_{K}( 1+ \| \xi^{\rm ini}\|_{\infty}).\]
 Therefore if we take $M\geq   c_K( 1+ \| \xi^{\rm ini}\|_{\infty})$, we conclude that  $\xi(t,x)$
coincides with the unique solution of \eqref{PDEintrAUX}, which is satisfied by the limit point of the sequence $\{g^N\}$.

\section{Technical proofs} \label{sec:proof-tech}
In this last section we prove Proposition \ref{prop:bound1} and Proposition \ref{prop:bound2}.

\subsection{Proof of Proposition \ref{prop:bound1}}
\label{ssec:proof1}
Let us prove the first estimate on $g^N$ given in Proposition \ref{prop:bound1}, namely \eqref{eq:firstestim}. \mm{Let $q \geqslant 2$.}

\bigskip
\noindent 
\textbf{Step 1}. From \eqref{equatiog1h} after a multiplication by $(\mathrm{I}-A)^{\mm{\alpha}/2}$ and by triangular inequality we have
\begin{align}
\Big\Vert \left(  \mathrm{I}-A\right)^{\mm{\alpha}/2}& g_{t}^{N}\Big\Vert _{\mathbb L^{\mm{p}}\left(  \mathbb{R}^{2}\right)  }  
\leq\left\Vert \left(  \mathrm{I}-A\right)  ^{\mm{\alpha}/2}e^{tA}g_{0}^{N}\right\Vert _{\mathbb L^{\mm{p}}\left( \mathbb{R}^{2}%
\right)  } \label{eq:first} \vphantom{\Bigg(}\\
&  +\int_{0}^{t}\left\Vert \left( \mathrm{I} -A\right)  ^{\mm{\alpha}/2}\nabla
e^{\left(  t-s\right)  A}\left(  V^{N} \ast\left(
F(K\ast g_{s}^{N}) S_{s}^{N}\right)  \right)  \right\Vert _{\mathbb L^{\mm{p}}\left(
\mathbb{R}^{2}\right)  }ds \label{eq:second} \vphantom{\Bigg(}\\
&  +\bigg\Vert \frac{{\sqrt{2\nu}}}{N}\sum_{i=1}^{N}\int_{0}^{t}\left(  \mathrm{I}-A\right)
^{\mm{\alpha}/2}\nabla e^{\left(  t-s\right)  A}\left(  
V^{N} \left(  \cdot-X_{s}^{i,N}\right)  \right)  dW_{s}%
^{i}\bigg\Vert _{\mathbb{L}^{\mm{p}}\left(  \mathbb{R}^{2}\right)  }. \label{eq:third}
\end{align}
We denote $H:=\LL^{\mm{p}}(\mathbb{R}^{2})$. Then 
\begin{align*}
\Big\Vert \left(  \mathrm{I}-A\right)^{\mm{\alpha}/2}& g_{t}^{N}\Big\Vert _{\mathbb L^{\mm{q}}\left( \Omega, H \right)  }  
\leq\left\Vert \left(  \mathrm{I}-A\right)  ^{\mm{\alpha}/2}e^{tA}g_{0}^{N}\right\Vert_{\mathbb L^{\mm{q}}\left( \Omega, H \right) }
  \vphantom{\Bigg(}\\
&  +\int_{0}^{t}\left\Vert \left( \mathrm{I} -A\right)  ^{\mm{\alpha}/2}\nabla
e^{\left(  t-s\right)  A}\left(  V^{N} \ast\left(
F(K\ast g_{s}^{N}) S_{s}^{N}\right)  \right)  \right\Vert_{\mathbb L^{\mm{q}}\left( \Omega, H \right) } ds  \vphantom{\Bigg(}\\
&  +\bigg\Vert \frac{{\sqrt{2\nu}}}{N}\sum_{i=1}^{N}\int_{0}^{t}\left(  \mathrm{I}-A\right)
^{\mm{\alpha}/2}\nabla e^{\left(  t-s\right)  A}\left(  
V^{N} \left(  \cdot-X_{s}^{i,N}\right)  \right)  dW_{s}%
^{i}\bigg\Vert_{\mathbb L^{\mm{q}}\left( \Omega, H \right) }. 
\end{align*}

\noindent \textbf{Step 2}. The first term \eqref{eq:first} can be estimated by
\[
\left\Vert \left(  \mathrm{I}-A\right) ^{\mm{\alpha}/2}e^{tA}g_{0}^{N}\right\Vert_{\mathbb L^{\mm{q}}\left( \Omega, H \right) }  
\leq \left\Vert \left(  \mathrm{I}-A\right)^{\mm{\alpha}/2} g_{0}^{N} \right\Vert_{\mathbb L^{\mm{q}}\left( \Omega, H \right) }  \leq   
C_{\mm{q}}.
\]
 The boundedness of $ g_{0}^{N}$ follows from Assumption \ref{assump}, item \textit{3.}

\bigskip

\noindent \textbf{Step 3}. Let us come to the second term \eqref{eq:second}:
\begin{align*}
  \int_{0}^{t}\big\Vert  &\left(   \mathrm{I}-A\right)  ^{\mm{\alpha}/2}\nabla
e^{\left(  t-s\right) A}\left(  V^{N} \ast\left(
F(K\ast g_{s}^{N}) S_{s}^{N}\right)  \right)  \big\Vert_{\mathbb L^{\mm{q}}\left( \Omega, H \right) } ds\\
&    \leq  C \; \int_{0}^{t} \bigg\{\big\Vert (   \mathrm{I}-A  )^{\mm{(1+\alpha)}/2}
e^{\mm{(t-s)} A}\big\Vert _{\LL^{\mm{p}}
\rightarrow \LL^{\mm{p}}  }\\ & \qquad \qquad\qquad \qquad \times \mm{\big\Vert 
 V^{N} \ast\left(  F(K\ast g_{s}^{N}) S_{s}^{N}\right)
  \big\Vert}_{\mathbb L^{\mm{q}}\left( \Omega, H \right) }\bigg\}ds.
\end{align*}
 We have 
\begin{equation*}
 \big\Vert \left(   \mathrm{I}-A\right)^{\mm{(1+\alpha)}/2}e^{\mm{(t-s)}A} \big\Vert _{\LL^{\mm{p}}  \rightarrow \LL^{\mm{p}} } \leq \frac{C_{\mm{\nu,T}}}{(t-s)^{\mm{(1+\alpha)}/2}} \mm{.}
\end{equation*}
On the other hand, for any $x \in \RR^{2}$, 
\[
\big\vert \left(  V^{N}\ast\left( F(K\ast g_{s}^{N})S_{s}^{N}\right)  \right)
\left(  x\right)  \big\vert \leq  \left\Vert F(K\ast g_{s}^{N}) \right\Vert _{\infty}   \big|V^{N} \ast S_{s}^{N}  \left(  x\right)\big|
\leq M  \; \big|g_{s}^{N}\left(  x\right)\big|.
\]
\mm{Hence,
\[\big\Vert  V^{N} \ast\left(  F(K\ast g_{s}^{N}) S_{s}^{N}\right)
 \big\Vert_{\mathbb L^{\mm{q}}\left( \Omega, H \right) }  \leqslant M \big\|g_s^N\big\|_{\LL^q(\Omega,H)} \leqslant C_M \big\|\left(  \mathrm{I}-A\right)  ^{\mm{\alpha}/2} g_s^N \big\|_{\LL^q(\Omega,H)}. \]

 }
%
%
\mm{To summarize, we have proved}
\begin{multline*}
\int_{0}^{t}\left\Vert \left(  \mathrm{I}-A\right)  ^{\mm{\alpha}/2}\nabla
e^{\left(  t-s\right)  A}\left(   V^{N} \ast\left(
F(K\ast g_{s}^{N}) S_{s}^{N}\right)  \right)  \right\Vert_{\mathbb L^{\mm{q}}\left( \Omega, H \right) }ds \\
 \leq \ C_{ \mm{\nu, M,T}}  \ \int_{0}^{t} \ (t-s)^{\mm{(1+\alpha)}/2} \ \left\Vert  \left(  \mathrm{I}-A\right)  ^{\mm{\alpha}/2}  g_{s}^{N}   \right\Vert_{\mathbb L^{\mm{q}}\left( \Omega, H \right) } ds .
\end{multline*}
This bounds the second term. \mm{Recall that $\alpha < 1$ therefore $(t-s)^{-(1+\alpha)/2}$ is integrable.}

\bigskip

\noindent\textbf{Step 4}. The estimate of the third term \eqref{eq:third} is quite tricky and we
postpone it to Lemma \ref{lemma est stoch term} below, see \eqref{unif estimate on stochastic term}. Collecting the three
bounds together, we get%
\[
\left\Vert \left(  \mathrm{I}-A\right)  ^{\mm{\alpha}/2} g_{t}^{N}\right\Vert_{\mathbb L^{\mm{q}}\left( \Omega, H \right) }
\leq   C_{\mm{q,T}} + C_{\mm{\nu, M,T}} \ \int_{0}^{t}   \   
(t-s)^{\mm{(1+\alpha)}/2}   \left\Vert \left(  \mathrm{I}-A\right)^{\mm{\alpha}/2}g_{s}^{N}\right\Vert_{\mathbb L^{\mm{q}}\left( \Omega, H \right) }  ds.
\]
We may apply  Gronwall's Lemma we  deduce
\[
\left\Vert \left(  \mathrm{I}-A\right) ^{\mm{\alpha}/2} g_{t}^{N}\right\Vert_{\mathbb L^{\mm{q}}\left( \Omega, H \right) }
\leq   C_{\mm{q,T, M, \nu}},\]
and Proposition \ref{prop:bound1} follows.

\begin{lemma}
\label{lemma est stoch term}  We assume Assumption \ref{assump}. \mm{Let $q\geqslant 2$.} Then there exists a constant $C_{\mm{q},T}>0$ such that for all $t \in [0,T]$,
\begin{equation}
\bigg\Vert \frac{1}{N}\sum_{i=1}^{N}\int_{0}^{t}\left(  \mathrm{I}-A\right)
^{\mm{\alpha}/2}\nabla e^{\left(  t-s\right)  A}\left( 
V^{N} \left(  \cdot-X_{s}^{i,N}\right)  \right)  dW_{s}%
^{i}\bigg\Vert_{\mathbb L^{\mm{q}}\left( \Omega, \LL^{\mm{p}}(\mathbb{R}^{2}) \right) }^{\mm{q}}\leq
C_{\mm{q,T}}. \label{unif estimate on stochastic term}%
\end{equation}

\end{lemma}

\begin{proof} 
\mm{From Sobolev embeddings we have
\begin{multline*}
\bigg\Vert \frac{1}{N}\sum_{i=1}^{N}\int_{0}^{t}\left(  \mathrm{I}-A\right)
^{\mm{\alpha}/2}\nabla e^{\left(  t-s\right)  A}\left( 
V^{N} \left(  \cdot-X_{s}^{i,N}\right)  \right)  dW_{s}%
^{i}\bigg\Vert_{\mathbb L^{\mm{q}}\left( \Omega, \mathbb{L}^\mm{p}(\RR^2) \right) }^{\mm{q}}
\\ \leqslant C \bigg\Vert \frac{1}{N}\sum_{i=1}^{N}\int_{0}^{t}\left(  \mathrm{I}-A\right)
^{\mm{(1+\alpha-\frac2p)}/2}\nabla e^{\left(  t-s\right)  A}\left( 
V^{N} \left(  \cdot-X_{s}^{i,N}\right)  \right)  dW_{s}%
^{i}\bigg\Vert_{\mathbb L^{\mm{q}}\left( \Omega, \mathbb{L}^\mm{2}(\RR^2) \right) }^{\mm{q}}.
\end{multline*}}
 From the Burkholder-Davis-Gundy inequality (see \cite{vanNeVerWeis} for instance)
we obtain
\begin{multline*}
\bigg\Vert \frac{1}{N}\sum_{i=1}^{N}\int_{0}^{t}\left(  \mathrm{I}-A\right)
^{\mm{(1+\alpha-\frac2p)}/2}\nabla e^{\left(  t-s\right)  A}\left( 
V^{N} \left(  \cdot-X_{s}^{i,N}\right)  \right)  dW_{s}%
^{i}\bigg\Vert_{\mathbb L^{\mm{q}}\left( \Omega,\mm{ \mathbb{L}^2(\RR^2)} \right) }^{\mm{q}}
\\
\leq C_{\mm{q}} \; \EE\bigg[    \frac{1}{N^{2}}\sum_{i=1}^{N}\int_{0}^{t} \left\Vert\left(  \mathrm{I}-A\right)
^{\mm{(1+\alpha-\frac2p)}/2}\nabla e^{\left(  t-s\right)  A}\left( 
V^{N} \left(  \cdot-X_{s}^{i,N}\right)  \right) \right\Vert_{\mm{ \mathbb{L}^2(\RR^2)}}^{2} ds\bigg]^{\mm{q}/2}.
\end{multline*}
Moreover, we can estimate 
\begin{align}
\frac{1}{N^{2}} \int_{\mathbb{R}^{2}}  \sum_{i=1}
^{N}\int_{0}^{t}&\left\vert \left(  \left(  \mathrm{I}-A\right)^{\mm{(1+\alpha-\frac2p)}/2}\nabla
e^{\left(  t-s\right)  A}\left( V^{N} \left(
\cdot-X_{s}^{i,N}\right)  \right)  \right)  \left(  x\right)  \right\vert
^{2}ds dx  \notag\\
& =\frac{1}{N} \int_{0}^{t}   \left\Vert   \left(   \mathrm{I}-A\right)^{\mm{(1+\alpha-\frac2p)}/2}\nabla
e^{\left(  t-s\right)  A} V^{N} \right\Vert_{\mm{ \mathbb{L}^2(\RR^2)}}^{2} \ ds  \notag
\\ & 
=\frac{1}{N} \int_{0}^{t}   \left\Vert   \left(   \mathrm{I}-A\right)^{-\delta/2}\nabla
e^{\left(  t-s\right)  A} \left(   \mathrm{I}-A\right)^{(\mm{1+\alpha-\frac2p}+\delta)/2} V^{N} \right\Vert_{\mm{ \mathbb{L}^2(\RR^2)}}^{2} \ ds  \notag
\\
& \leq \frac{1}{N} \int_{0}^{t} \frac{1}{(t-s)^{1-\delta}}  \left\Vert V^{N} \right\Vert_{\mm{1+\alpha-\frac2p}+\delta,2}^{2} \ ds  \notag \\
\label{eqq}
& \leq C_{\mm{T,\delta}}\;  N^{\beta(2+2\delta+2\mm{\alpha}\mm{+2-\frac4p})-1}. \vphantom{\Bigg)}
\end{align}
Therefore \eqref{eqq} is bounded by some constant $C_{\mm{q,T}}$ if we take $\beta<\frac{1}{\mm{4+ 2\alpha-\frac4p}}$, $\delta$ close enough to zero.  This provides the bound of the lemma.
\end{proof}

\subsection{Proof of Proposition \ref{prop:bound2}}
\label{ssec:proof2}

Let us now prove the second estimate on $g^N$ given in Proposition \ref{prop:bound2}, namely \eqref{eq:secondestim}. \mm{Let $q'\geqslant 2$. }
In this proof we use the fact that $\LL^2(\RR^{2}) \subset \HH_{2}^{-2}$ with continuous embedding, and that the linear operator $\Delta$ is bounded from $\LL^2(\RR^{2})$ to $\HH_{2}^{-2}$. 

\mm{Let us first recall that, from interpolation, from Proposition \ref{prop:bound1} and using the fact that $\|g_t^N\|_{\mathbb{L}^1(\RR^2)}=1$, we have: for any $\theta \in (0,1)$,
\begin{equation}\label{eq:new} 
\EE\Big[ \big\|g_t^N\big\|_{0,2}^{q'}  \Big]\leqslant \EE\Big[ \big\|g_t^N\big\|_{0,p}^{\theta q'} \; \;  \big\|g_t^N\big\|_{\LL^1(\RR^2)}^{(1-\theta)q'} \Big] \leqslant \EE\Big[  \big\|g_t^N\big\|_{0,p}^{\theta q'} \Big] \leqslant C_{q',T}.
\end{equation}
}
We \mm{then} observe that 
\begin{align*}
 g_{t}^{N}(x) - g_{s}^{N}(x) &  =
\int_{s}^{t}\left\langle S_{r}%
^{N}, (K\ast F(g_{r}^{N})) \ \nabla  V^{N}  \left(  x-\cdot\right)  \right\rangle dr +  \nu \int_{s}^{t}  \Delta g_{r}^{N} (x)dr \\
& \quad +\frac{{\sqrt{2\nu}}}{N}\sum_{i=1}^{N}\int_{s}^{t}\nabla \left( V^{N}\right)  \left(  x-X_{r}^{i,N}\right)dW_{r}^{i} .
\end{align*}
Therefore we have 
\begin{align}
\EE \Big[ \big\Vert 
 g_{t}^{N}(x) - & g_{s}^{N}(x) \big\Vert_{-2,2}^{\mm{q'}} \Big] \notag \\ & \leq
 (t-s)^{\mm{q'}-1}  \int_{s}^{t}  \EE\Big[\left\Vert \left\langle S_{r}%
^{N},  F(K\ast g_{r}^{N}) \ \nabla  V^{N}  \left(  x-\cdot\right)  \right\rangle \right\Vert_{-2,2}^{\mm{q'}} \Big]  dr \label{eq:t1}\\
& \quad +  (t-s)^{\mm{q'}-1} \ \frac{1}{2}\int_{s}^{t}  \EE \Big[\left\Vert  \Delta g_{r}^{N} (x)  \right\Vert_{-2,2}^{\mm{q'}}\Big] \ dr \label{eq:t2}\\
& \quad + \EE \bigg[\bigg\Vert \frac{{\sqrt{2\nu}}}{N}\sum_{i=1}^{N}\int_{s}^{t}\nabla \left( V^{N}\right)  \left(  x-X_{r}^{i,N}\right)
dW_{r}^{i}   \bigg\Vert_{-2,2}^{\mm{q'}}\bigg]. \label{eq:t3}
\end{align}
To estimate the first term \eqref{eq:t1} we observe first that 
\begin{align}
\EE\Big[\big\Vert \left\langle S_{r}%
^{N}, F(K\ast g_{r}^{N})\ \nabla  V^{N}  \left(  x-\cdot\right)  \right\rangle \big\Vert_{-2,2}^{\mm{q'}} \Big] &= \EE\Big[\left\Vert  \nabla ( S_{r}^{N} F(K\ast g_{r}^{N})\ast  V^{N})  \right\Vert_{-2,2}^{\mm{q'}} \Big]
\notag \\ & 
\leq \EE\Big[\left\Vert  (S_{r}^{N}  F(K\ast g_{r}^{N}) \ast  V^{N}  \right\Vert_{-1,2}^{\mm{q'}}\Big]. \notag
\\ & 
\leq   C_{M} \EE\Big[\left\Vert  g_{t}^{N}  \right\Vert_{\LL^2(\RR^{2})}^{\mm{q'}}\Big]\leq C. \label{b0}
\end{align}
Moreover, for the second term \eqref{eq:t2} we write
\begin{equation}\label{b2}
 \EE\Big[\left\Vert  \Delta g_{r}^{N}   \right\Vert_{-2,2}^{\mm{q'}}\Big]\leq  C
\EE\Big[\left\Vert  g_{r}^{N}  \right\Vert_{{\LL^2(\RR^{2})}}^{\mm{q'}}\Big]\leq C_{\mm{q',T}},   
\end{equation}
\mm{from \eqref{eq:new}.} Finally  we bound the last term \eqref{eq:t3}:
\begin{multline*}
 \EE \bigg[\bigg\Vert   \frac{1}{N}\sum_{i=1}^{N}\int_{s}^{t}\nabla \left( V^{N}\right)  \left(  x-X_{r}^{i,N}\right)
dW_{r}^{i}   \bigg\Vert_{-2,2}^{\mm{q'}}\bigg]
\\ 
 \leq C_{\mm{q'}} \EE \bigg[   \frac{1}{N^{2}} \sum_{i=1}^{N}\int_{s}^{t} \Big\Vert  \nabla \left( V^{N}\right)  \left(  x-X_{r}^{i,N}\right)   \Big\Vert_{-2,2}^{\mm{q'}}  dr  \bigg]^{\mm{q'}/2}
\end{multline*}
and we observe  that 
\begin{multline*}
\frac{1}{N^{2}} \int_\RR  \ \sum_{i=1}^{N}\int_{s}^{t}  \Big\|(\mathrm{I}-A)^{-1}\; \nabla \left( V^{N}\right)  \left(  x-X_{r}^{i,N}\right)\Big\|^{2}dr   dx \\    =(t-s) \frac{1}{N} \big\| V^{N}  \big\|_{-1,2}^{2}   \leq (t-s) \frac{1}{N} \big\|V^N \big\|_{0,2}^2   \leq C N^{2\beta-1} (t-s) \leq C (t-s). \vphantom{\Big(}
\end{multline*}
In order to conclude the lemma, we need to divide \eqref{eq:t1}--\eqref{eq:t3} by $|t-s|^{1+\mm{q'}\gamma}$. From the previous estimates, we always get a term of the form $|t-s|^\varepsilon$ with $\varepsilon <1$ (using the assumption $\gamma < \frac12$). 
\newpage

\appendix

\section{More general initial data}
\label{app}

Assume that the initial vorticity $\xi^{\mathrm{ini}}$ has variable sign. Define%
\[
\xi_{+}^{\mathrm{ini}}:=\xi^{\mathrm{ini}}\vee0,\qquad\xi_{-}^{\mathrm{ini}%
}:=\left(  -\xi^{\mathrm{ini}}\right)  \vee0
\]
and
\[
\Gamma_{\pm}:=\int\xi_{\pm}^{\mathrm{ini}}\left(  x\right)  dx>0
\]
(they are finite, since we  assume $\xi^{\mathrm{ini}}\in \LL^{1}$). Let
$\{ X_{0}^{i,\pm}\}  _{i\in\mathbb{N}}$ be a double sequence of
random variables in $\mathbb{R}^{2}$ such that, for the empirical
measures%
\[
S_{0}^{N,\pm}:=\frac{\Gamma_{\pm}}{N}\sum_{i=1}^{N}\delta_{X_{0}^{i,\pm}}%
\]
one has, for some $\alpha>\mm{\frac2p}$, for any $\mm{q} >0$,%
\begin{equation}
\sup_{N\in\mathbb{N}}\mathbb{E}\left[  \big\Vert V^{N}\ast S_{0}^{N,\pm
}\big\Vert _{\mm{\alpha,p}}^{\mm{q}}\right]  \ <\infty\label{initial cond twosigns}%
\end{equation}
and the two sequences of measures $\{S_{0}^{N,\pm}\}_{N\in\mathbb{N}}$ weakly
converge to the initial measures $\xi_{\pm}^{\mathrm{ini}}(\cdot)dx$, as
$N\rightarrow\infty$, in probability: \textit{i.e.}
\begin{equation}
S_{0}^{N,\pm}\xrightarrow[N\rightarrow\infty]{}\xi_{\pm
}^{\mathrm{ini}}(\cdot)dx\qquad\text{in probability.}%
\label{initial conv twosigns}%
\end{equation}
Consider the system of PDEs, given for $x\in\mathbb{R}^{2},t>0$ by%
\begin{equation}\label{eq:PDEsystem}
\begin{aligned}
\partial_{t}\xi^{+}+u\;\sca\; \nabla\xi^{+}  &  =\nu\Delta\xi^{+},\\
\partial_{t}\xi^{-}+u\;\sca\; \nabla\xi^{-}  &  =\nu\Delta\xi^{-},\\
u  &  =K\ast\left(  \xi^{+}-\xi^{-}\right) \\
\xi^{+}|_{t=0}  &  =\xi_{0}^{+},\qquad\xi^{-}|_{t=0}=\xi_{0}^{-}.
\end{aligned} \end{equation}
It is not difficult to prove the same results of existence and
uniqueness as the ones obtained for the usual Navier-Stokes equations
\begin{equation}
\partial_{t}\xi+u\;\sca\; \nabla\xi=\nu\Delta\xi,\qquad x\in\mathbb{R}%
^{2},t>0. \label{Vorty2}%
\end{equation}
Moreover, if $\left(  \xi^{+},\xi^{-}\right)  $ is a solution of the system,
then $\xi=\xi^{+}-\xi^{-}$ is a solution of (\ref{Vorty2}); if $\xi$ is a
solution of (\ref{Vorty2}) and $\xi^{+}$ is a solution of the first equation
of the system, with $u=K\ast\xi$, then $\xi^{-}=\xi^{+}-\xi$ is a solution of
the second equation of the system, and $u=K\ast\left(  \xi^{+}-\xi^{-}\right)
$ holds. In this sense the system and (\ref{Vorty2}) are equivalent.

\bigskip

Let $\{  W_{t}^{i,\pm}\}_{i\in\mathbb{N}}$ be a family of
independent $\mathbb{R}^{2}$-valued Brownian motions, defined on the same
probability space as $\{ X_{0}^{i,\pm}\} _{i\in\mathbb{N}}$ and
independent of them. Given $N\in\mathbb{N}$ consider particles with positions $\{
X_{t}^{i,N,\pm}\}_{i\in\mathbb{N}}$ satisfying%
\begin{align*}
dX_{t}^{i,N,+}    = & \; F\bigg( \Gamma_{+} \; \frac{1}{N}\sum_{k=1}^{N}(K\ast V^{N}%
)(X_{t}^{i,N,+}-X_{t}^{k,N,+}) \\ & \qquad -  \Gamma_{-}\; \frac{1}{N}\sum_{k=1}^{N}(K\ast V^{N}%
)(X_{t}^{i,N,+}-X_{t}^{k,N,-})\bigg)  \;dt\\
&  +\sqrt{2\nu}\;dW_{t}^{i,+} \vphantom{\bigg(}\\
dX_{t}^{i,N,-}   =& \; F\bigg(  \Gamma_{+}\; \frac{1}{N}\sum_{k=1}^{N}(K\ast V^{N}%
)(X_{t}^{i,N,-}-X_{t}^{k,N,+}) \\ & \qquad -\Gamma_{-}\; \frac{1}{N}\sum_{k=1}^{N}(K\ast V^{N}%
)(X_{t}^{i,N,-}-X_{t}^{k,N,-})\bigg)  \;dt\\
&  +\sqrt{2\nu}\;dW_{t}^{i,-}\vphantom{\bigg(}
\end{align*}
with initial conditions $\{  X_{0}^{i,\pm}\} _{i\in\mathbb{N}}$.
Consider the associated empirical measures%
\[
S_{t}^{N,\pm}:=\frac{\Gamma_{\pm}}{N}\sum_{i=1}^{N}\delta_{X_{t}^{i,\pm}}%
\]
and empirical densities%
\[
g_{t}^{N,\pm}:=V^{N}\ast S_{t}^{N,\pm}.
\]

\begin{theorem}\label{Thm 1bis} Assume
on $V$, $\beta$, $\alpha$, \mm{$p$}, the same conditions of Assumption \ref{assump}
and in addition assume (\ref{initial cond twosigns}) and
(\ref{initial conv twosigns}). Consider the particle system $\{
X_{t}^{i,\pm}\}  _{i\in\mathbb{N}}$ with the parameter $M$ which
satisfies
\begin{equation}
M\geqslant c_{K}\left(  1+\Vert\xi^{\mathrm{ini}}\Vert_{\mathbb{L}^{\infty}%
}\right).  \label{eq:assMbis}
\end{equation} Then, for every $\eta
\in(\mm{\frac2p},\alpha)$, the sequence of processes $\{  g_{t}^{N,+},g_{t}%
^{N,-}\}  $ converges in probability with respect to the

\begin{itemize}
\item weak topology of $\Big(\mathbb{L}^{2}\left(  [0,T]\;;\;\mathbb{H}%
_{\mm{p}}^{\alpha}(\mathbb{R}^{2})\right)\Big)  ^{2}$,

\item strong topology of $\Big(C\big(  [0,T]\;;\;\mm{\mathbb{W}_{\mathrm{loc}}^{\eta,p}}(\mathbb{R}^{2})\big) \Big) ^{2}$,
\end{itemize}
as $N\rightarrow\infty$, to the unique weak solution of the PDE system \eqref{eq:PDEsystem}
and thus $g_{t}^{N,+}-g_{t}^{N,-}$ converges, in the same topologies, to the
unique weak solution of the PDE (\ref{Vorty2}).
\end{theorem}

We do not repeat the full proof in this case but only sketch the main points.
The empirical measures satisfy%
\begin{align*}
\left\langle S_{t}^{N,+},\phi\right\rangle    =  \left\langle S_{0}^{N,+}%
,\phi\right\rangle &+\int_{0}^{t}\left\langle S_{s}^{N,+},F\left(  K\ast
V^{N}\ast S_{s}^{N,+}-K\ast V^{N}\ast S_{s}^{N,-}\right) \;\sca\; \nabla
\phi\right\rangle ds\\
&  +\nu\int_{0}^{t}\left\langle S_{s}^{N,+},\Delta\phi\right\rangle
ds+\frac{{\sqrt{2\nu}}}{N}\sum_{i=1}^{N}\int_{0}^{t}\nabla\phi(X_{s}^{i,N,+}%
) \; \sca\;dW_{s}^{i,+}%
\end{align*}%
and 
\begin{align*}
\left\langle S_{t}^{N,-},\phi\right\rangle    =  \left\langle S_{0}^{N,-}%
,\phi\right\rangle& +\int_{0}^{t}\left\langle S_{s}^{N,-},F\left(  K\ast
V^{N}\ast S_{s}^{N,+}-K\ast V^{N}\ast S_{s}^{N,-}\right) \;\sca\; \nabla
\phi\right\rangle ds\\
&  +\nu\int_{0}^{t}\left\langle S_{s}^{N,-},\Delta\phi\right\rangle
ds+\frac{{\sqrt{2\nu}}}{N}\sum_{i=1}^{N}\int_{0}^{t}\nabla\phi(X_{s}^{i,N,-}%
) \; \sca\;dW_{s}^{i,-}%
\end{align*}
and the empirical densities satisfy%
\begin{align*}
g_{t}^{N,+}(x)    =  g_{0}^{N,+}(x) & +\int_{0}^{t}\left\langle S_{s}%
^{N,+},F(K\ast g_{s}^{N,+}-K\ast g_{s}^{N,-})\;\sca\;\nabla V^{N}\left(
x-\cdot\right)  \right\rangle ds\\
&  +\nu\int_{0}^{t}\Delta g_{s}^{N,+}(x)ds+\frac{{\sqrt{2\nu}}}{N}\sum_{i=1}^{N}\int%
_{0}^{t}\nabla V^{N}\left(  x-X_{s}^{i,N,+}\right) \sca\;dW_{s}^{i,+}%
\end{align*}%
and 
\begin{align*}
g_{t}^{N,-}(x)    =  g_{0}^{N,-}(x)&+\int_{0}^{t}\left\langle S_{s}%
^{N,-},F(K\ast g_{s}^{N,+}-K\ast g_{s}^{N,-})\;\sca\;\nabla V^{N}\left(
x-\cdot\right)  \right\rangle ds\\
&  +\nu\int_{0}^{t}\Delta g_{s}^{N,-}(x)ds+\frac{{\sqrt{2\nu}}}{N}\sum_{i=1}^{N}\int%
_{0}^{t}\nabla V^{N}\left(  x-X_{s}^{i,N,-}\right)  \sca\;dW_{s}^{i,-}%
\end{align*}
which in mild form are%
\begin{align*}
g_{t}^{N,+}   =  e^{tA}{g}_{0}^{N,+}&+\int_{0}^{t}\nabla e^{\left(  t-s\right)
A}\left(  V^{N}\ast\left(  F(K\ast g_{s}^{N,+}-K\ast g_{s}^{N,-})S_{s}%
^{N,+}\right)  \right)  ds\\
&  +\frac{{\sqrt{2\nu}}}{N}\sum_{i=1}^{N}\int_{0}^{t}e^{\left(  t-s\right)  A}\left(
\nabla V^{N}\left(  \cdot-X_{s}^{i,N,+}\right)  \right) \sca\;dW_{s}^{i,+}%
\end{align*}%
and
\begin{align*}
g_{t}^{N,-}    =  e^{tA}{g}_{0}^{N,-}&+\int_{0}^{t}\nabla e^{\left(  t-s\right)
A}\left(  V^{N}\ast\left(  F(K\ast g_{s}^{N,+}-K\ast g_{s}^{N,-})S_{s}%
^{N,-}\right)  \right)  ds\\
&  +\frac{{\sqrt{2\nu}}}{N}\sum_{i=1}^{N}\int_{0}^{t}e^{\left(  t-s\right)  A}\left(
\nabla V^{N}\left(  \cdot-X_{s}^{i,N,-}\right)  \right)  \sca\;dW_{s}%
^{i,-}.
\end{align*}
The proof of the estimate%
\[
\mathbb{E}\left[  \left\Vert \left(  \mathrm{I}-A\right)  ^{\mm{\alpha}
/2}g_{t}^{N,+}\right\Vert _{\mathbb{L}^{\mm{p}}\left(  \mathbb{R}^{2}\right)  }%
^{\mm{q}}\right]  \leq C_{\mm{T,M,\nu,q}}%
\]
(and similarly for $g_{t}^{N,-}$) is similar to the case of a single sign (Proposition \ref{prop:bound1}): precisely, the estimates on%
\[
\left\Vert \left(  \mathrm{I}-A\right)  ^{\mm{\alpha}/2}e^{tA}g_{0}%
^{N,+}\right\Vert _{\mathbb{L}^{\mm{p}}\left(  \mathbb{R}^{2}\right)  }%
\]
and on
\[
\left\Vert \frac{1}{N}\sum_{i=1}^{N}\int_{0}^{t}\left(  \mathrm{I}-A\right)
^{\mm{\alpha}/2}\nabla e^{\left(  t-s\right)  A}\left(  V^{N}\left(  \cdot
-X_{s}^{i,N,+}\right)  \right)  dW_{s}^{i,+}\right\Vert _{\mathbb{L}%
^{\mm{p}}\left(  \mathbb{R}^{2}\right)  }%
\]
are obviously the same. But also the middle term can be studied in the same
way: one has%
\begin{multline*} \left\vert \left(  V^{N}\ast\left(  F\left(  K\ast g_{s}^{N,+}-K\ast
g_{s}^{N,-}\right)  S_{s}^{N,+}\right)  \right)  \left(  x\right)  \right\vert\vphantom{\Big(}
\\
 \leq\left\Vert F(K\ast g_{s}^{N,+}-K\ast g_{s}^{N,-})\right\Vert _{\infty
}\left\vert V^{N}\ast S_{s}^{N,+}\left(  x\right)  \right\vert \leq
M\;\left\vert g_{s}^{N,+}\left(  x\right)  \right\vert \vphantom{\Big(}
\end{multline*}
\mm{and the same conclusion follows, using Gronwall's Lemma.}

Moreover, in the proof of%
\[
\mathbb{E}\bigg[  \int_{0}^{T}\int_{0}^{T}\frac{\big\Vert g_{t}^{N,+}%
-g_{s}^{N,+}\big\Vert _{-2,2}^{\mm{q'}}}{|t-s|^{1+\mm{q'}\gamma}}\;ds\;dt\bigg]  \leq
C_{T,M,\nu,\mm{q'}}%
\]
(and similarly for $g_{t}^{N,-}$) the only part which \textit{a priori} may change is%
\begin{align*}
  \mathbb{E}\Big[  & \left\Vert \left\langle S_{r}^{N,+},F(K\ast g_{s}%
^{N,+}-K\ast g_{s}^{N,-})\ \nabla V^{N}\left(  x-\cdot\right)  \right\rangle
\right\Vert _{-2,2}^{\mm{q'}}\Big]\\
&  =\mathbb{E}\left[  \left\Vert \nabla(S_{r}^{N,+}F(K\ast g_{s}^{N,+}-K\ast
g_{s}^{N,-})\ast V^{N})\right\Vert _{-2,2}^{\mm{q'}}\right] \\
&  \leq\mathbb{E}\left[  \left\Vert S_{r}^{N,+}F(K\ast g_{s}^{N,+}-K\ast
g_{s}^{N,-})\ast V^{N}\right\Vert _{-1,2}^{\mm{q'}}\right] \\
&  \leq C_{M}\mathbb{E}\left[  \left\Vert g_{t}^{N}\right\Vert _{\LL^2}%
^{\mm{q'}}\right]  \leq C,
\end{align*}
so in fact this part remains the same. Then one can apply the same
arguments of tightness (see Section \ref{subsect compactness}).

In the passage to the limit (Section \ref{sec:pass}), the arguments are similar. We use the weak
formulation
\begin{align*}
\left\langle g_{t}^{N,+},\phi\right\rangle  &  =\left\langle g_{0}^{N,+}%
,\phi\right\rangle +\int_{0}^{t}\left\langle S_{s}^{N,+},F(K\ast g_{s}%
^{N,+}-K\ast g_{s}^{N,-})\;\sca\;\nabla(V^{N}\ast\phi)\right\rangle ds\\
&  \quad+\nu\int_{0}^{t}\left\langle g_{s}^{N,+},\Delta\phi\right\rangle
(x)ds+\frac{{\sqrt{2\nu}}}{N}\sum_{i=1}^{N}\int_{0}^{t}\nabla(V^{N}\ast\phi)\left(
X_{s}^{i,N,+}\right)  \;\sca\;dW_{s}^{i,+}%
\end{align*}
where $\phi$ is a smooth test function with compact support. The only
difficult step is proving that
\begin{multline*}
  \lim_{N\rightarrow\infty}\int_{0}^{t}\left\langle S_{s}^{N,+},F\big(K\ast
g_{s}^{N,+}-K\ast g_{s}^{N,-}\big) \;\sca\;\nabla(V^{N}\ast\phi)\right\rangle ds\\
  =\int_{0}^{t}\int_{\mathbb{R}^{2}}\xi^{+}\left(  s,x\right)  F\Big(K\ast\left(
\xi^{+}\left(  s,x\right)  -\xi^{-}\left(  s,x\right)  \right)  \Big) \;\sca\;\nabla\phi\left(  x\right)  ds.
\end{multline*}
The proof is analogous to the case of a single sign;\ let us recall the main
steps. After application of Skorohod's Theorem, due to the a.s.~convergence of
$g_{s}^{N,\pm}$ to $\xi^{\pm}$ in $\mathbb{L}_{w}^{2}\left(  [0,T]\; ;\;\mathbb{H}%
_{\mm{p}}^{\mm{\alpha}}\right)  $ and called $f_{s}^{N,\pm}:=K\ast g_{s}^{N,\pm}$,
thanks to the properties of the Biot-Savart operator we get that (a.s.)
$f_{s}^{N,\pm}$ converge to $K\ast\xi^{\pm}$ in $\mathbb{L}_{w}^{2}\left(
[0,T]\; ;\;\mathbb{L}^{p}\right)  $;\ and moreover $f_{s}^{N,\pm}$ are (a.s.)
bounded in $\mathbb{L}^{2}\left( [ 0,T]\; ;\;C^{\widetilde{\eta}}\right)  $. The
last property is used to prove that%
\begin{multline*}
\Big|F(f_{s}^{N,+}-f_{s}^{N,-})(x)\;\sca\;\nabla(V^{N}\ast\phi)(x)
\\
-\left(  F(f_{s}^{N,+}-f_{s}^{N,-})\;\sca\;\nabla(V^{N}\ast\phi)\right)  \ast
V^{N}(x)\Big|\leq C/N^{\widetilde{\eta}\beta}.
\end{multline*}
From this the other steps are easier and equal to the one-sign case.

\section{Uniform convergence}
\label{appB}

All past papers dealing with particle approximation of $2d$ Navier-Stokes
equations prove weak convergence of the empirical measures $S_{t}^{N}$ to the
probability law $\xi_{t}dx$. The novelty here is that we prove a stronger
convergence, namely the convergence in suitable function spaces of the
mollified empirical measure $g_{t}^{N}:=V^{N}\ast S_{t}^{N}$. Consider for
instance the property of uniform convergence in space on compact sets ($\LL^{2}$
in time). It is not possible to deduce this result from the weak convergence
$S_{t}^{N}\rightharpoonup\xi_{t}dx$ (see below). If one only knows that
$S_{t}^{N}\rightharpoonup\xi_{t}dx$, and one considers a classical kernel
smoothing algorithm $\theta_{\epsilon_{N}}\ast S_{t}^{N}$ to approximate the
profile $\xi_{t}$ by means of $S_{t}^{N}$, it is not clear how to choose
$\theta_{\epsilon_{N}}$ in such a way to have uniform convergence of
$\theta_{\epsilon_{N}}\ast S_{t}^{N}$ to $\xi_{t}$. The method described in
this paper indicates a strategy for a better particle approximation of
solutions to $2d$ Navier-Stokes equations.

Let us understand more closely the strength of the uniform convergence. It is
a strong indication that the particles are distributed quite uniformly in
space, they do not have too much concentration, aggregation. Let us make this
remark more quantitative.

\begin{proposition}
Assume  that the probability density $V$ has the property%
\[
h\mathbf{1}_{\left[  -r,r\right]  ^{2}}\leq V
\]
for some constants $h,r>0$. Assume that for some $\left(  t,\omega,R\right)  $
we have%
\[
\sup_{\left\vert x\right\vert \leq R}g_{t}^{N}\left(  x\right)  \leq C.
\]
Then%
\begin{equation}
\sup_{\left\vert x\right\vert \leq R}\mathrm{Card}\left\{  i=1,...,N \; ; \; X_{t}^{i,N}%
\in\left[  x-\frac{r}{N^{\beta}},x+\frac{r}{N^{\beta}}\right]  ^{2}\right\}
\leq\frac{C}{h}N^{1-2\beta}. \label{number of points}%
\end{equation}

\end{proposition}

\begin{proof}
The result simply follows from the inequalities%
\[
hN^{2\beta}\mathbf{1}_{\left[  -\frac{r}{N^{\beta}},\frac{r}{N^{\beta}}\right]  ^{2}%
}\leq V^{N}%
\]
and thus%
\[
h\frac{\mathrm{Card}\left\{  i=1,...,N:X_{t}^{i,N}\in\left[  x-\frac{r}{N^{\beta}%
},x+\frac{r}{N^{\beta}}\right]  ^{2}\right\}  }{N^{1-2\beta}}\leq g_{t}%
^{N}\left(  x\right)  .
\]
\end{proof}

Let us give a heuristic interpretation of the previous result.

First consider the case of points $X_{t}^{i,N}$ geometrically uniform in
$\left[  -\frac{1}{2},\frac{1}{2}\right]  ^{2}$:\ we consider a uniform grid
in $\left[  -\frac{1}{2},\frac{1}{2}\right]  ^{2}$ of side length $\frac
{1}{N^{1/2}}$, hence with (roughly)\ $N$ grid points, and we put one particle
per grid point. 

In a box $\left[  x-\frac{r}{N^{\beta}},x+\frac{r}{N^{\beta}%
}\right]  ^{2}\subset\left[  -\frac{1}{2},\frac{1}{2}\right]  ^{2}$ we have
(approximatively)\
\[
\left(  \frac{\frac{2r}{N^{\beta}}}{\frac{1}{N^{1/2}}}\right)  ^{2}%
=4r^{2}N^{1-2\beta}%
\]
points. This is exactly estimate (\ref{number of points}).

\bigskip

Now, let us break this uniformity. Divide $\left[  -\frac{1}{2},\frac{1}%
{2}\right]  ^{2}$ in two sets:  \[Q=\left[  0,\frac{1}{N^{1/2}}\right]  ^{2}, \qquad \text{and} \qquad Q^{c}:=\left[  -\frac{1}{2},\frac{1}{2}\right]  ^{2}\backslash Q.\] Call
$X_{t}^{i}$ the first $N-N^{1-\beta}$ particles, $\widetilde{X}_{t}^{i}$ the
last $N^{1-\beta}$ ones. Put the $N-N^{1-\beta}$ particles $X_{t}^{i}$ in the
grid points of $Q^{c}$, no more than one particle per grid point. Put the
$N^{1-\beta}$ particles $\widetilde{X}_{t}^{i}$ in $Q_{0}=\left[  0,\frac
{r}{N^{\beta}}\right]  ^{2}$, which is contained in $Q$ (for large $N$), being
$\beta<\frac{1}{4}$. 

Then, all cubes $Q_{x}:=\left[  x-\frac{r}{N^{\beta}},x+\frac
{r}{N^{\beta}}\right]  ^{2} \subset Q^{c}$ contain, as above, at most
$4r^{2}N^{1-2\beta}$ particles. But $Q_{0}$ contains $N^{1-\beta}$ particles,
much more than $4r^{2}N^{1-2\beta}$ for large $N$. This is an example of
configuration which does not fulfill estimate (\ref{number of points}). For
such configuration, however, we still have%
\[
\left\langle S_{t}^{N},\phi\right\rangle \rightarrow\left\langle \xi_{t}%
,\phi\right\rangle
\]
for all continuous bounded $\phi$. Let us only check that this is true for
$\phi=\mathbf{1}_{\mathcal{O}}$ when $\mathcal{O}$ is an open set (it is not continuous but it is a good
heuristic indication). We have
\begin{align*}
\left\langle S_{t}^{N},\phi\right\rangle  & =\left\langle S_{t}^{N},\mathbf{1}_{\mathcal{O}\cap
Q}\right\rangle +\left\langle S_{t}^{N},\mathbf{1}_{\mathcal{O}\cap Q^{c}}\right\rangle
\leq\left\langle S_{t}^{N},\mathbf{1}_{Q}\right\rangle +\left\langle S_{t}^{N},\mathbf{1}_{\mathcal{O}\cap
Q^{c}}\right\rangle \\
& \sim\frac{1}{N}N^{1-\beta}+\left\vert \mathcal{O}\right\vert \xrightarrow[{N\rightarrow
\infty}]{}\left\vert \mathcal{O}\right\vert =\left\langle \xi_{t}%
,\phi\right\rangle
\end{align*}
where $\xi_{t}=\mathbf{1}$. The excessive concentration in $Q_{0}$ does not prevent
weak convergence but it is not allowed by the stronger convergence proved here.

\bibliographystyle{siamplain}
\bibliography{references}
\end{document}

%% file: ex_shared.tex
\usepackage{lipsum}
\usepackage{amsfonts}
\usepackage{graphicx}
\usepackage{epstopdf}
\usepackage{algorithmic}
\ifpdf
  \DeclareGraphicsExtensions{.eps,.pdf,.png,.jpg}
\else
  \DeclareGraphicsExtensions{.eps}
\fi

\newsiamremark{remark}{Remark}
\newsiamremark{hypothesis}{Hypothesis}
\crefname{hypothesis}{Hypothesis}{Hypotheses}
\newsiamthm{claim}{Claim}
\newsiamthm{assumption}{Assumption}

\headers{Uniform  approximation of 2$d$ Navier-Stokes equation}{F. Flandoli, C. Olivera, and M. Simon}

\title{Uniform  approximation of 2$d$ Navier-Stokes equation  by stochastic interacting particle systems\thanks{\funding{ C. O.~is partially supported by FAPESP by the grant 2018/15258-7	 and by CNPq
		by the grant 426747/2018-6. This project has received funding from the CNRS-FAPESP cooperation, grant n\textsuperscript{o}PRC2726. M.S.~also thanks Labex CEMPI (ANR-11-LABX-0007-01), and the ANR grant MICMOV (ANR-19-CE40-0012) of the French National Research Agency (ANR), and finally  the European Research Council (ERC) under  the European Union's Horizon 2020 research and innovative program (grant agreement  n\textsuperscript{o}715734). }}}

\author{Franco Flandoli\thanks{Scuola Normale Superiore, Pisa,
Italy
  (\email{franco.flandoli@sns.it}).}
\and Christian
Olivera\thanks{Departamento de Matem\'{a}tica, Universidade Estadual de
Campinas UNICAMP, Brazil 
  (\email{colivera@ime.unicamp.br}).}
\and Marielle Simon\thanks{Inria, Univ. Lille, CNRS, UMR 8524 - Laboratoire Paul Painlev\'e, F-59000 Lille, France (\email{marielle.simon@inria.fr})}}

\usepackage{amsopn}
\usepackage{amssymb}
\usepackage{mathabx}

\newcommand{\RR}{\mathbb{R}}
\newcommand{\sca}{\,\begin{picture}(-1,1)(-1,-3)\circle*{2.5}\end{picture}\ }
\newcommand{\LL}{\mathbb{L}}
\newcommand{\NN}{\mathbb{N}}
\newcommand{\HH}{\mathbb{H}}
\newcommand{\PP}{\mathbb{P}}
\newcommand{\EE}{\mathbb{E}}

\newcommand{\mm}[1]{{#1}}
